\documentclass[11pt, oneside, reqno]{amsart}  	

\usepackage[colorinlistoftodos]{todonotes}
\usepackage[T1]{fontenc}
\usepackage[utf8]{inputenc}

\usepackage{graphicx}

\usepackage{parskip}

\usepackage{geometry}                		
\usepackage{amssymb}
\usepackage{amscd, xcolor, changepage,cancel, mathtools}
\usepackage{faktor}
\usepackage{bbold}
\usepackage{comment}
\usepackage{centernot}

\usepackage{microtype}
\usepackage{xurl}

\usepackage[breaklinks=true]{hyperref}
\hypersetup{
    colorlinks=true,
    linkcolor=blue,
    citecolor=blue,
    filecolor=blue,      
    urlcolor=blue,
}
\usepackage[most]{tcolorbox}

\usepackage[backend=bibtex,style=alphabetic,sorting=anyt]{biblatex} 
\makeatletter
\def\@tocline#1#2#3#4#5#6#7{\relax
  \ifnum #1>\c@tocdepth 
  \else
    \par \addpenalty\@secpenalty\addvspace{-5pt}%
    \begingroup \hyphenpenalty\@M
    \@ifempty{#4}{%
      \@tempdima\csname r@tocindent\number#1\endcsname\relax
    }{%
      \@tempdima#4\relax
    }%
    \parindent\z@ \leftskip#3\relax \advance\leftskip\@tempdima\relax
    \rightskip\@pnumwidth plus4em \parfillskip-\@pnumwidth
    #5\leavevmode\hskip-\@tempdima
      \ifcase #1
       \or\or  \hskip 1em \or \hskip 2em \else \hskip 3em \fi%
      #6\nobreak\relax
    \hfill\hbox to\@pnumwidth{\@tocpagenum{#7}}\par
    \nobreak
    \endgroup
  \fi}
\makeatother

\definecolor{dimgray}{rgb}{0.41, 0.41, 0.41}
\definecolor{afb}{rgb}{0.36, 0.54, 0.66}
\definecolor{bncs}{rgb}{0.0, 0.53, 0.74}
\definecolor{tangerine}{rgb}{0.95, 0.52, 0.0}
\definecolor{nice_green}{rgb}{0.0, 0.5, 0.0}
\definecolor{teal}{rgb}{0.0, 0.5, 0.5}

\definecolor{bdf}{rgb}{0.19,0.55, 0.91}
\definecolor{blush}{rgb}{0.87, 0.36, 0.51}

\usepackage{scalerel} 
\DeclareMathOperator*{\Expectation}{\scalerel*{\mathbb{E}}{\textstyle\sum}}

\DeclareMathOperator{\sgn}{sgn}
\DeclareMathOperator{\supp}{supp}

\DeclareMathOperator{\Stab}{Stab}
\DeclareMathOperator{\diag}{diag}
\DeclareMathOperator{\off}{off diag}

\newcommand{\mS}{\mathcal{S}}

\newcommand{\E}{\mathbb{E}}
\newcommand{\F}{\mathbb{F}}
\newcommand{\C}{\mathbb{C}}
\newcommand{\N}{\mathbb{N}}
\newcommand{\Z}{\mathbb{Z}}
\newcommand{\Prob}{\mathbb{P}}

\newcommand{\la}{\langle}
\newcommand{\ra}{\rangle}

\newcommand{\norm}[2]{\left\| #1 \right\|_{#2}}

\newcommand{\ap}{\alpha}
\newcommand{\de}{\delta}

\newcommand{\ep}{\epsilon}

\newcommand{\wh}{\widehat}

\newcommand{\lp}{\left(}
\newcommand{\rp}{\right)}

\newcommand{\lf}{\left|}
\newcommand{\rf}{\right|}

\newcommand{\Ib}{I^{c}} 
\newcommand{\Sb}{S^{c}} 

\numberwithin{equation}{section}
\allowdisplaybreaks 

\newtheorem{theorem}{Theorem}[section]

\newtheorem{proposition}[theorem]{Proposition} 
\newtheorem{lemma}[theorem]{Lemma}
\newtheorem*{sameCharRestatement}{Lemma \ref*{lem:sameChar}}
\newtheorem*{diffCharRestatement}{Lemma \ref*{lem:diffChar}}

\theoremstyle{definition}
\newtheorem{definition}[theorem]{Definition}

\theoremstyle{remark}
\newtheorem{remark}[theorem]{Remark}

\title{Reasonable bounds for Restricted Polynomial Progressions in $\F^n_p$}

\author[A. Fragkos]{Anastasios Fragkos}
\address{Georgia Institute of Technology}
\email{anastasiosfragkos@gatech.edu}
\author[A. Burgin]{Alex Burgin}
\address{Georgia Institute of Technology}
\email{aburgin6@gatech.edu}
\author[D. Kruzel]{David Kruzel} 
\address{Georgia Institute of Technology}
\email{dkruzel3@gatech.edu}
\author[M. Lacey]{Michael Lacey} 
\address{Georgia Institute of Technology}
\email{lacey@math.gatech.edu}
\author[A. Stokolosa]{Amelia Stokolosa}
\address{Georgia Institute of Technology}
\email{astokolosa3@gatech.edu}
\begin{document}
\maketitle

\begin{abstract}
    For each integer $k\geq2$, we obtain parameters $M_k=\exp(O(k\log k))$ and $P_k=\exp\!\big(\exp(O(k^2\log k))\big)$ satisfying the following: for every prime $p\geq P_k$, we show that if $A \subseteq \F^n_p$ does not contain a restricted polynomial progression of the form 
    \begin{align*}
        x, x+\begin{pmatrix}
        y_1 \\ y_2 \\ \vdots \\ y_n
    \end{pmatrix}, x+\begin{pmatrix}
        y_1^2 \\ y_2^2 \\ \vdots \\ y_n^2
    \end{pmatrix}, \ldots, x+\begin{pmatrix}
        y_1^{k-1} \\ y_2^{k-1} \\ \vdots \\ y_n^{k-1}
    \end{pmatrix},
    \end{align*}
    where $x \in \F^n_p$ and $y \in \{0,2, 3, \ldots, M_k\}^n \backslash \{\vec{0} \}$, then $A$ must have density bounded above by $\frac{C}{(\log^{(k^{Ck})} n)^c}$, where $C, c>0$ only depend on $k$ and $p$. For $k\geq 4$, the best previously known bound came from the density Hales--Jewett theorem and was of inverse Ackermann type. The proof uses a recent inverse theorem for noise stability of Bhangale, Khot, Liu and Minzer, which arose in the study of constraint satisfaction problems. The main new ingredient is a proof that the underlying distribution admits no nontrivial abelian embedding. 
\end{abstract}


\tableofcontents

\section{Introduction}


\subsection{Polynomial Progressions} \ 

Bergelson and Leibman \cite{BL96} proved that every subset of the integers with positive upper density contains arbitrarily long polynomial progressions of the form $x, x+ P_1(y), \ldots, x+ P_k(y)$, where $y \in \Z \backslash \{0\}$ and $P_1, \ldots, P_k $ are integer polynomials with no constant term. Following Bergelson and Leibman's foundational work on the so-called polynomial Szemer\'edi theorem, considerable effort has been devoted to obtaining quantitative versions of this qualitative result. In the finite field setting, Bourgain and Chang \cite{Bourgain_Chang_17} first established a quantitative nonlinear Roth theorem using Fourier-analytic techniques and exponential sum estimates. They showed that a subset $A$ of a prime field $\F_p$ lacking $3$-term progressions of the form $x, x+y, x+y^2$ must have cardinality $|A| \ll p^{\frac{14}{15}}$. Subsequently, Peluse \cite{Pel_18_poly_roth} and Dong, Li, and Sawin \cite{DLS_20_poly_roth} obtained power savings bounds for more general $3$-term polynomial progressions $x, x+P_1(y), x+P_2(y)$ in $\F_p$, where $P_1, P_2 \in \Z[y]$ are linearly independent polynomials with zero constant terms. Both arguments rely on machinery borrowed from algebraic geometry. Given $k\geq 3$ linearly independent polynomials $P_1, \ldots, P_k$ vanishing at $0$, Peluse \cite{peluse_19} later proved power savings bounds for sets $A \subseteq \F_p$ lacking longer polynomial progressions of the form $x, x+ P_1(y), \ldots, x+P_k(y)$. More recently, Kuca \cite{Kuca_21} established a quantitative result for sets lacking polynomial progressions that are neither linearly independent, nor homogeneous of the same degree; namely, progressions of the form $x, x+y, \ldots, x+(m-1)y, x+y^m, \ldots, x+y^{m+k-1}$ for various choices of $k,m \in \N$. He later built upon a PET induction argument, a degree-lowering method, a quantitative concatenation result for averages of box norms over finite abelian groups, and a box norm smoothing technique to establish a multi-dimensional result in \cite{Kuca24_poly_gowers_ff, Kuca_multi_finite_distinct_24} extending a previously obtained power savings bound for sets lacking the polynomial configuration $(x_1, x_2), (x_1+ y, x_2), (x_1, x_2+ y^2)$ in $\F^2_p$ by Han, Lacey, and Yang in \cite{HanLaceyFanCorners}. 

In the integer setting, Prendiville \cite{Prendiville_homogeneous} first considered homogeneous polynomial progressions of the same degree. Together with Peluse \cite{Peluse_prend_nonlinear_roth, PP_nonlinear_roth_polylog}, it was shown that there exists $c>0$ such that if $A \subseteq [N]$ contains no non-trivial polynomial progression of the form $x, x+y, x+y^2$, then 
\begin{align*}
    |A| \ll \frac{N}{( \log N)^c}. 
\end{align*}
Peluse \cite{peluse_20} extended the result to longer progressions by showing that given polynomials $P_1, \ldots, P_m \in \Z[y]$ of distinct degrees with zero constant term, a set $A \subseteq [N]$ lacking polynomial progressions $x, x+ P_1(y), \ldots, x+ P_m(y)$ satisfies
\begin{align*}
    |A| \ll_{P_1, \ldots, P_m} \frac{N}{ (\log \log N)^{c_{P_1, \ldots, P_m} }}. 
\end{align*}
Shao and Wang \cite{shao_wang_2025_poly_log} recently improved the bound to 
\begin{align*}
    |A| \ll \frac{N}{( \log N)^c}. 
\end{align*}
Also recently, Leng \cite{Leng_24_complexity_one} obtained improved bounds for sets $A \subseteq \Z/ N\Z$ lacking $4$-term polynomial progressions of the form $x, x+ P(y), x+ Q(y), x+ P(y) + Q(y)$, where $P$ and $Q$ are linearly independent polynomials with $P(0) = Q(0) =0$. Leng showed that such sets have size $|A| \ll \frac{N}{\log^{(C)}N}$, for some absolute constant $C$. 

Related developments include corners with polynomial length \cite{KKL_24}, multidimensional polynomial patterns \cite{PPS_24_multi}, polynomial progressions in the primes \cite{Tao_Ziegler_08}, and polynomial progressions in topological fields \cite{KMPW_top_fields_24}. Other notable works on polynomial configurations include \cite{bergelson_best2025, peluse_sah_swawhney2023_shifted_roth, KKL_26_concatenation, Hong_Lim_25} and the references therein. 

\subsection{Restricted progressions} \

Bhangale, Khot, Liu, and Minzer \cite{CSP_6} recently resolved the \emph{restricted $3$-AP problem} in $\F^n_p$, which had appeared on Ben Green's list of open problems. A \textit{restricted $3$-AP} in $\F^n_p$ is a three-term arithmetic progression $x, x+a, x+2a$ with $x \in \F^n_p$ and $a \in \{0,1\}^n \backslash \{\vec{0}\}$. They show that if $A \subseteq \F^n_p$ has density $\frac{|A|}{p^n}\geq \frac{C_p}{(\log \log \log n)^{c_p}}$, then $A$ contains a restricted $3$-AP. 
To the best of our knowledge, the only previously known quantitative bounds came from the density Hales--Jewett theorem \cite{PolymathHalesJewett}, and they were of inverse Ackermann type, $\frac{|A|}{p^n} \leq O \lp \frac{1}{\sqrt{\log^* n}} \rp$. The result of \cite{CSP_6} was thus the first \emph{reasonable} bound for this problem, meaning one involving only a bounded number of iterated logarithms. In earlier work \cite{BKM-EffectiveRestricted}, the same group of authors obtained quantitative bounds for restricted progressions whose common difference lies in the larger set $\{0,1,2\}^n \backslash \{\vec{0} \}$. Shrinking the set of allowed differences to $\{0,1\}^n$ required an inverse theorem under weaker hypotheses, which was established in \cite{CSP_6}. Their argument rests on a stability result arising in the study of the complexity of constraint satisfaction problems. Other work on arithmetic progressions whose common difference is drawn from a sparse set includes that of Cook and Magyar \cite{Cook_Magyar_12}, who consider differences in an algebraic set given as a level set of a family of homogeneous polynomials, and the recent note of Conlon, Fox and Pham \cite{Conlon_Fox_Pham_restricted_2026} on arithmetic progressions with restricted differences.

In this paper we consider the polynomial analogue: progressions $x, x+y, x+y^2,\ldots,x+y^{k-1}$ in $\F_p^n$, where the powers are taken coordinatewise and the step $y$ is restricted to a product set $Y^n$ with $Y\subseteq \F_p$ a fixed finite set of small integers.

\subsection{Main result} \

Fix $k \geq 2$. Let $M(k) = \exp(O(k \log k ))$ and $P(k) = \exp \exp (O(k^2 \log k))$ be the functions provided by Proposition \ref{prop:uniformTorsionRank}, and set
\begin{align}\label{e;MkPk}
    M_k \coloneqq M(k-1), \qquad P_k \coloneqq \max\big\{P(k-1),\ M_k^{\,k-1}\big\}+1, \qquad Y \coloneqq \{0,2,3, \ldots, M_k\}.
\end{align}
These choices satisfy $M_k=\exp(O(k\log k))$ and $ P_k=\exp\!\big(\exp(O(k^2\log k))\big)$. 
Indeed, $M_k^{k-1}=\exp(O(k^2\log k))$, so the additional term in the definition of $P_k$ does not increase its stated growth rate.

Sections \ref{sec:reduction}--\ref{sec:diffChar} are phrased in terms of the largest power $\ell$ in the progression, which here is $\ell = k-1$. This is why $M$ and $P$ are evaluated at $k-1$. Let $p\geq P_k$ be prime. We say that a set $A\subseteq \F_p^n$ has a \textit{restricted $k$-term polynomial progression} if $A$ contains $k$ vectors
\begin{equation}
    x, \ x+ P_1 (y), \ldots, x+ P_{k-1}(y),
\end{equation}
where $P_j(y) = (y_1^j, \ldots, y_n^j)$ for $1 \leq j \leq k-1$ and $y \in Y^n \backslash \{\vec{0}\}$.

The choice of $Y$ is deliberate. The value $y=1$ is excluded because it collapses the progression: $x+1^j = x+1$ for every $j$, so a step with coordinates in $\{0,1\}$ produces only two distinct points. The value $y = 0$ is included for two reasons. First, the resulting diagonal tuples $(x,\ldots,x)$ in the support of the underlying distribution are what allow us to eliminate one of the functions when ruling out abelian embeddings. Second, freezing a coordinate by taking $y_i=0$ is what makes the configuration stable under passing to coordinate fibers (Lemma \ref{l;fiber}).

Define the iterated exponential and logarithm recursively by $\exp^{(0)} x \coloneqq x$, $\exp^{(j+1)} x \coloneqq \exp\big(\exp^{(j)} x\big)$, and $\log^{(0)} x \coloneqq x$, $\log^{(j+1)} x \coloneqq \log\big(\log^{(j)} x\big)$, so that $\log^{(j)}$ is the inverse of $\exp^{(j)}$. We use $\log^{(j)} x$ only for $x > \exp^{(j)}(1)$, which ensures that it is defined and exceeds $1$.

\begin{theorem}\label{main thm}
    Let $k\geq2$, and let $M_k=\exp(O(k\log k))$ and $P_k=\exp\!\big(\exp(O(k^2\log k))\big)$ be the parameters obtained from Proposition \ref{prop:uniformTorsionRank} as in \eqref{e;MkPk}. Set $Y=\{0,2,3,\ldots,M_k\}$.
    There is an integer $K=k^{O(k)}$ such that for every prime $p\geq P_k$ there exist $C=C(k,p)>0$ and $c=c(k,p)>0$ with the following property. If $n>\exp^{(K)}(1)$ and $A\subseteq\F_p^n$ satisfies
    \begin{align*}
        \frac{|A|}{p^n}\geq\frac{C}{\big(\log^{(K)}n\big)^c},
    \end{align*}
    then $A$ contains a nontrivial restricted $k$-term polynomial progression $x,x+P_1(y),\ldots,x+P_{k-1}(y)$ with $y\in Y^n\setminus\{\vec0\}$. Moreover, the $k$ points of any such progression are pairwise distinct.
\end{theorem}

For $k=2$ and $k=3$, bounds at least this strong were already available (see Section \ref{sec:existing}). The new content of Theorem \ref{main thm} is the range $k\geq 4$. There, to our knowledge, the best previously available bound came from the density Hales--Jewett theorem and was of inverse Ackermann type. Theorem \ref{main thm} replaces it with a bound involving $k^{O(k)}$ iterated logarithms.

\begin{remark}[The conditions on $p$ and $M_k$]\label{r;conditions}
    We record where each condition in \eqref{e;MkPk} is used.
    \begin{enumerate}
        \item[(a)] $M_k = M(k-1)\geq k$ (by the parameter choice in the quantitative proof of Proposition \ref{prop:uniformTorsionRank}) and $p>P(k-1)$. Together with $p>M_k$ below, these are the hypotheses of Proposition \ref{prop:noAbelianEmbeddings}, applied with $\ell = k-1$. The condition $M_k \geq k$ is used in the equal-characteristic case (Lemma \ref{lem:sameChar}), and $p>P(k-1)$ in the case of different characteristic (Lemma \ref{lem:diffChar}).
        \item[(b)] $p>M_k$, which follows from $p > M_k^{k-1}$ since $M_k\geq k\geq 2$ by (a). The elements of $Y$ are then distinct modulo $p$. Hence the tuples $(x,x+y,\ldots,x+y^{k-1})$, $x\in\F_p$, $y\in Y$, are pairwise distinct, as used in the computation of the atoms and marginals of $\mu$ in Lemma \ref{l;large exp}. The same condition makes the Vandermonde determinant in Lemma \ref{lem:sameChar} nonzero in $\F_p$.
        \item[(c)] $p>M_k^{k-1}$. For $y\in\{2,\ldots,M_k\}$ we have $1<y<y^2<\cdots<y^{k-1}<p$, so $0, y, y^2,\ldots,y^{k-1}$ are distinct in $\F_p$. If $y\in Y^n\setminus\{\vec 0\}$, some coordinate has $y_i\in\{2,\ldots,M_k\}$. Comparing the points $x, x+P_1(y),\ldots,x+P_{k-1}(y)$ in that coordinate shows that they are pairwise distinct. Without this condition a ``nontrivial'' progression could collapse to fewer than $k$ points.
    \end{enumerate}
\end{remark}

\subsection{Standard methods are not well suited to this question} \

The standard Gowers norm approach to polynomial configurations begins with repeated applications of a van der Corput inequality, leading to a  suitable Gowers or box norm of one of the functions involved. 
This strategy relies heavily on averaging over a full set of shifts, which is far from available in the restricted case.
In some settings, one can appeal to a normalized majorant satisfying appropriate pseudorandomness conditions, which must generally be established by separate analytic estimates. In our setting, however, the normalized weight associated with the product restriction $y\in Y^n$ is very far from being pseudorandom.
Its Cartesian-product structure produces a Fourier transform of Riesz product type, which has large Fourier coefficients supported on only a few coordinates, uniformly in the dimension $n$. Concretely, the weight is $\nu(y):=(p/|Y|)^n\mathbf{1}_{Y^n}(y)$, normalized so that $\mathbb{E}_{y\in \mathbb{F}_p^n}\nu(y)=1$, and $\widehat{\nu}$ factors as $\widehat{\nu}(\xi)=\prod_{i=1}^n \widehat{\nu_0}(\xi_i)$ with $\widehat{\nu_0}(\xi_i):=\frac{1}{|Y|}\sum_{t\in Y}e_p(-\xi_i t)$. As $Y$ is a nonempty proper subset of $\mathbb{F}_p$, the indicator $\mathbf{1}_Y$ is not constant, so $c:=|\widehat{\nu_0}(a)|>0$ for some $a\neq 0$. Testing the frequencies $\xi^{(i)}:=ae_i$, for which $|\widehat{\nu}(\xi^{(i)})|=c$, gives
\begin{align*}
    \|\nu-1\|_{U^2}^4=\sum_{\xi\neq 0}|\widehat{\nu}(\xi)|^4\geq \sum_{i=1}^n |\widehat{\nu}(\xi^{(i)})|^4 =nc^4,
\end{align*}
which grows with $n$.

Moreover, suppose one could establish and apply a suitable inverse theorem for the Gowers norms. Such inverse theorems, in the discrete and continuous cases, deduce that a function correlates with a non-classical polynomial, or a nilsequence. 
 A density increment is obtained on a subspace or a level set of a polynomial. 
But such sets need not preserve the restricted configurations. 
By contrast, the distribution defining our configurations is a tensor power $\mu^{\otimes n}$, and we use an inverse theorem adapted to this tensor-product structure. Once one knows that $\mu$ admits no nontrivial abelian embedding, this theorem yields a density increment on a coordinate fiber. Such fibers preserve the form of the restricted configuration, and therefore allow the argument to be iterated.

\subsection{The inverse theorem of Bhangale, Khot, Liu and Minzer} \label{sec:inverseThm}

Our proof of Theorem \ref{main thm} relies on an inverse theorem of Bhangale, Khot, Liu and Minzer \cite[Theorem 1]{CSP_7}. It is part of a larger research program \cite{CSP-I, CSP-II, CSP-III, CSP_4, CSP-V, CSP_6} on the complexity of satisfiable \emph{constraint satisfaction problems} (CSPs), with a recent survey \cite{minzer2026lensabelianembeddings}. 
Briefly, the subject is this. 
Let $\mathcal{P}$ be a collection of predicates $P:\Sigma^m \rightarrow \{0,1\}$ on a finite alphabet $\Sigma$. An instance of $\mathcal{P}$-CSP consists of variables $x_1,\ldots,x_N$ and constraints, each of which applies some $P \in \mathcal{P}$ to an $m$-tuple of the variables. The instance is satisfiable if some assignment of values in $\Sigma$ to the variables satisfies every constraint. A basic algorithmic question is to decide whether an instance is satisfiable, or to determine the largest fraction of its constraints that can be satisfied simultaneously. Understanding when this fraction can be approximated efficiently for satisfiable instances leads, through dictatorship testing, to questions about correlations $\E_{\mu^{\otimes n}}\prod_i f_i(x_i)$ of functions under product distributions. Theorem \ref{t;inverse} is an answer to one such question. The theorem applies to a notably broad class of distributions $\mu$, extending earlier work of Mossel \cite{Moss10} on so-called \emph{connected} distributions.\footnote{A distribution $\mu$ is said to be \textit{connected} if the graph with vertex set $\supp \mu$, in which two vertices are adjacent whenever they differ in exactly one coordinate, is connected.}

To state the inverse theorem, we require the concept of an abelian embedding.

\begin{definition}\label{d;abelianEmbedding}
    Let $\Sigma_1,\ldots , \Sigma_k$ be finite alphabets and $\mS\subseteq \prod_{i=1}^k \Sigma_i$. For an abelian group $(G,+)$, we say that maps $\sigma_i:\Sigma_i\rightarrow G$, $i=1,\ldots,k$, form an \textit{abelian embedding} of $\mS$ if \begin{align*}
        (a_1,\ldots,a_k)\in \mS\implies \sum_{i=1}^k\sigma_i(a_i)=0_G.
    \end{align*} The embedding is \textit{trivial} if every $\sigma_i$ is constant, and \textit{nontrivial} otherwise. A distribution $\mu$ over $\prod_{i=1}^k\Sigma_i$ \textit{admits a nontrivial abelian embedding} if its support $\supp\mu$ has a nontrivial abelian embedding into some abelian group, and \textit{admits no nontrivial abelian embedding} otherwise.
\end{definition}

\begin{theorem}[{\cite[Theorem 1]{CSP_7}}]
\label{t;inverse}
Let \( k \) be a positive integer and let \( \mu \)  be a distribution over
\( \Sigma_1 \times \cdots \times \Sigma_k \) that admits no nontrivial abelian embedding, and in which the probability of each atom is at least \( \gamma \).
 Then, for every \( \epsilon>0 \)   there are \( \delta\coloneqq \delta(\gamma,\epsilon) >0 \) and \( n_{\ast} = n_{\ast}(\gamma,\epsilon) \)
  such that for every \( n\geq n_{\ast} \), if 1-bounded
functions \( f_i \colon \Sigma_{i}^n \to \mathbb{C} \)  satisfy
\begin{equation} \label{B;bigger}
    \lf \Expectation_{(x_1, \  \ldots, \  x_k) \sim \mu^{\otimes n}}
    \prod_{i=1}^k f_i(x_i) \rf \geq \epsilon,
\end{equation}
then \(\textup{Stab}_{1-\delta} (f_i) > \delta\) for all \( i=1, \ldots,k \), where the stability of \( f_i \) is taken with respect to the marginal \( \mu_i \) of \( \mu \) on \( \Sigma_i \).
 Quantitatively, we can take $k^{O(k)}$ iterates of the exponential function; that is,
\begin{equation}
    \delta \coloneqq \exp(- \exp^{(k^{O(k)})}(\epsilon^{-O_\gamma(1)}) ).
\end{equation}
\end{theorem}

Three features of Theorem \ref{t;inverse} make it well suited to our problem. First, its only structural hypothesis is the absence of abelian embeddings, a purely algebraic condition on $\supp\mu$. No pseudorandomness of a majorant is needed, which sidesteps the obstruction described above. Second, its conclusion concerns noise stability, a notion defined coordinate by coordinate. Noise stability converts directly into a large Efron--Stein weight at low degree, and then, by a random restriction, into a density increment on a coordinate fiber (Lemmas \ref{l;stability} and \ref{p;restriction}). Such fibers are exactly the substructures on which our configuration reproduces itself. Third, the bounds are reasonable: they involve a number of iterated exponentials depending only on $k$.

Stability results from this program have found several combinatorial applications. Besides the restricted $3$-AP problem \cite{CSP_6} discussed above, they yield reasonable bounds in the density Hales--Jewett theorem for combinatorial lines of length three \cite{bhangale2024reasonableboundscombinatoriallines}. The results of \cite{CSP-I, CSP-II, CSP-III, CSP_4, CSP-V} also contain related stability theorems for shorter, ternary distributions.

\subsection{Outline of the proof} \

The argument is a density increment on coordinate fibers.

\emph{Step 1: a counting distribution.} Let $\mu$ be the distribution on $\F_p^k$ that, with probability $\frac12$, outputs a uniformly random diagonal tuple $(x,\ldots,x)$, and otherwise outputs $(x,x+y,\ldots,x+y^{k-1})$ with $x\in\F_p$ and $y\in\{2,\ldots,M_k\}$ chosen uniformly.
 All marginals of $\mu$ are uniform. 
 A sample from $\mu^{\otimes n}$, on \( (\F_p^{n})^{k} \), is a restricted progression, and it is trivial exactly when every coordinate is diagonal, which has probability $2^{-n}$. 
 So if $A\subseteq \F_p^{n}$ has density $\ap$ and contains no nontrivial progression, then $\E_{\mu^{\otimes n}}\prod_{j=0}^{k-1} 1_A(x_j) = \ap 2^{-n}$, which is far below the value $\ap^k$, which would hold if \( A \) were selected at random. 
  Writing $1_A = \ap + f_A$ and expanding, some correlation of the form $\E\big[f_A(x_j)\prod_{i\neq j}f_i(x_i)\big]$ is at least $c_k\ap^k$ in absolute value (Lemma \ref{l;large exp}).

\emph{Step 2: no abelian embeddings (Sections \ref{sec:reduction}--\ref{sec:diffChar}).} To apply Theorem \ref{t;inverse}, we show that $\mu$ admits no nontrivial abelian embedding. A general reduction shows that it suffices to rule out embeddings $f_0,\ldots,f_{k-1}:\F_p\to\Z/q\Z$ with $q$ prime. The diagonal tuples force $f_0 = -\sum_{j\ge1}f_j$, and subtracting leaves one linear equation in the translation operator $T$ for each $y\in\{2,\ldots,M_k\}$. When $q=p$, the operator $T$ is unipotent, and a Vandermonde determinant shows that these equations force every $f_j$ to be constant. When $q\neq p$, we diagonalize $T$ and must show that a matrix $A_{\ell,M}(\lambda) = (\lambda^{y^j}-1)_{y,j}$ has full rank at every root of unity $\lambda$ of order $p$ in characteristic $q$. We give two proofs of this rank assertion, one qualitative and the other quantitative, both beginning in characteristic zero. The qualitative proof uses Kronecker's theorem and a Noetherian argument, while the quantitative proof uses a determinant estimate and a result of Dimitrov. Both then pass to finite characteristic using a B\'ezout identity and a separate degree argument for the exceptional characteristics. In the quantitative proof, a bound on the integer multiplier in the B\'ezout identity gives the required bound on $P_k$.

\emph{Step 3: from stability to a density increment (Section \ref{sec:proofMain}).} Theorem \ref{t;inverse} gives $\Stab_{1-\de}(f_A) > \de$ with $\de$ depending only on $\ap$, $k$ and $p$. Hence $f_A$ has Efron--Stein weight at least $\de/2$ at degree $d\approx \de^{-1}\log(1/\de)$. A random restriction then produces a coordinate fiber of dimension at least $n/(2d)$ on which $A$ has density at least $\ap + \de/32$.

\emph{Step 4: iteration.} Because $0 \in Y$, a fiber of a set without nontrivial progressions again has none (Lemma \ref{l;fiber}). The increment can therefore be iterated. Since densities cannot exceed $1$, there are at most $32/\de$ steps, and the dimension must fall below the threshold of Theorem \ref{t;inverse} before that happens. This bounds $n$ by a tower of height $k^{O(k)}$ in $1/\ap$, which is equivalent to Theorem \ref{main thm}.

\subsection{Existing quantitative results for this configuration}\label{sec:existing}

\subsubsection{The cases $k\leq 3$} For $k\in \{2,3\}$, results at least as strong as Theorem \ref{main thm} were already known.

\begin{itemize}
    \item For $k=2$, a density Hales--Jewett (DHJ) argument, described below for $k\geq 4$, applies. For $k=2$, DHJ reduces to Sperner's theorem, which has very good quantitative bounds, and yields \begin{align*}
        \frac{|A|}{p^n}\ll n^{-1/2}.
    \end{align*}

    \item For $k=3$, one can appeal directly to a ternary stability theorem of Bhangale, Khot, Liu and Minzer \cite[Theorem 3]{CSP_6}, the engine behind their solution of the restricted $3$-AP problem, in which the restricted digits are simply $\{0,1\}$. Its bound involves three iterated logarithms of $n$, rather than a tower of height $k^{O(k)}$.

    The hypothesis of \cite[Theorem 3]{CSP_6} is that the projection of the distribution onto each \emph{pair} of coordinates is connected. That is, for each pair of coordinates, the graph whose vertices are the support of the corresponding $2$-ary marginal, two vertices being adjacent when they differ in exactly one coordinate, is connected. This is weaker than connectedness of the support itself, in the sense of the footnote in Section \ref{sec:inverseThm}. The stronger property in fact fails here: no two points of the supports below differ in exactly one coordinate, since agreeing in two coordinates forces the two points to be equal.

    Let $p$ be odd and let $\nu$ be the uniform distribution on
    \begin{align*}
        R:=\{(x,x+t,x+t^2):x\in \mathbb{F}_p,\ t\in \{0,2\}\},
    \end{align*}
    which is contained in the support of the distribution $\mu$ of Section \ref{sec:proofMain}. A set with no nontrivial restricted progression with step in $Y^n$ has, in particular, none with step in $\{0,2\}^n$, so it suffices to apply \cite[Theorem 3]{CSP_6} to $\nu$.

    The projections of $\nu$ are connected. Project onto the first two coordinates, obtaining the support $\{(x,x) : x \in \F_p\}\cup\{(x,x+2) : x \in \F_p\}$. For each $x$, the vertices $(x,x)$ and $(x,x+2)$ differ only in the second coordinate, and $(x,x+2)$ and $(x+2,x+2)$ differ only in the first; hence $(x,x)$ and $(x+2,x+2)$ are joined by a path of length two. Since $p$ is odd, $2$ generates $(\F_p,+)$, so iterating this joins $(x,x)$ to $(y,y)$ for every $y$, and every remaining vertex $(x,x+2)$ is adjacent to $(x,x)$. The projection onto the second and third coordinates is the same graph after the change of variable $u=x+2$, and the projection onto the first and third coordinates is the analogous graph with $2$ replaced by $4$, which also generates $(\F_p,+)$ because $p$ is odd. Hence
    \begin{align*}
        \frac{|A|}{p^n}\ll \frac{1}{(\log \log \log n)^{c}}.
    \end{align*}

    It is the connectedness hypothesis that is used here, not the absence of abelian embeddings: the distribution $\nu$ does admit a nontrivial abelian embedding, for instance $f_0(x)=x$, $f_1(x)=-2x$, $f_2(x)=x$, as one checks on $(x,x,x)$ and on $(x,x+2,x+4)$. For the full step set $\{0,2,\ldots,M_3\}$ there is no nontrivial abelian embedding, by Proposition \ref{prop:noAbelianEmbeddings} with $\ell = 2$, but Theorem \ref{t;inverse} would then only give the weaker tower bound of Theorem \ref{main thm}.

\end{itemize}

\subsubsection{The case $k\geq 4$}
For $k\geq 4$, the best bound for this configuration known to us follows from the density Hales--Jewett theorem. It gives only a very weak, inverse Ackermann-type decay in the density, as explained in the introduction to \cite{PolymathHalesJewett}.
The argument is as follows. 
Let $S=\{0,2,2^2,\cdots, 2^{k-1}\}\subseteq \mathbb{F}_p$, with $p$ sufficiently large in terms of $k$, so that these elements are distinct. We view this as a list of $|S|$ symbols.  Take $A\subset \mathbb{F}_p^n $ of density $\alpha$. Since every $a\in A$ belongs to exactly $|S|^n$ of the translates $(z+S^n)_{z\in \mathbb{F}_p^n}$, averaging over $z\in \mathbb{F}_p^n$ gives that \begin{align*}
    p^{-n}\sum_{z\in \mathbb{F}_p^n}\frac{|A\cap (z+S^n)|}{|S|^n}=\alpha.
\end{align*} Consequently, there exists some $z\in \mathbb{F}_p^n$ such that \begin{align}\label{eq:intro1}
    |A\cap (z+S^n)|\geq \alpha |S|^n.
\end{align} Set \begin{align*}
    B:=\{s\in S^n:z+s\in A\}\subset S^n,
\end{align*} then by (\ref{eq:intro1}) $|B|/|S|^n\geq \alpha$. Let $\text{DHJ}(k,\alpha)$ denote the threshold given by Density Hales-Jewett: for $n\geq \text{DHJ}(k,\alpha)$, any set $B\subset S^n$ of density at least $\alpha$ contains a combinatorial line of length $k$ in $B$. This means, if $B$ has density at least $\alpha$, then there is a nonempty `wildcard set' $W\subset [n]$ and fixed symbols $s_i\in S$, $i\not\in W$, such that for every $t\in S$, the word $w(t)=(w(t)_1,\cdots,w(t)_n)$ with \begin{align*}
    w(t)_i=\begin{cases}
        t, & i\in W\\
        s_i & i\not\in W
    \end{cases} 
\end{align*} belongs to $B$. Thus, for $n\geq \text{DHJ}(k,\alpha)$, \begin{align}\label{eq:intro2}
    z+w(t)\in A,\quad \forall t\in S.
\end{align} We now define $x,y\in \mathbb{F}_p^n$ by \begin{align*}
    x_i:=\begin{cases}
        z_i,\quad i\in W\\
        z_i+s_i,\quad i\not\in W,
    \end{cases}\quad y_i:=\begin{cases}
        2,\quad i\in W\\
        0,\quad i\not\in W.
    \end{cases}
\end{align*} Then, since $W\neq \varnothing$, $y\in \{0,2\}^n\setminus \{\vec 0\}$. For $1\leq j\leq k-1$, \begin{align*}
    x_i+y_i^j=\begin{cases}
        z_i+2^j,& i\in W\\
        z_i+s_i,& i\not\in W
    \end{cases}
\end{align*} and so \begin{align*}
    x=z+w(0),\quad x+P_j(y)=z+w(2^j), \quad 1\leq j\leq k-1,
\end{align*} where $(P_j(y))_i:=y_i^j$. By (\ref{eq:intro2}), one has then that $x,x+P_1(y),x+P_2(y),\cdots, x+P_{k-1}(y)$ are all in $A$. This is an instance of our configuration, and so if $n\geq \text{DHJ}(k,\alpha)$ our configuration appears.

\subsection{Use of AI} 
Insights from conversations with generative AI were used to develop several aspects of the paper. This included developing ideas, testing hypotheses, advice on organization, positioning, and novelty, and assisting with revisions. It was used to make final edits of the paper. The authors take full responsibility for the content of the paper. 

This paper is formalized in Lean 4, by an AI agent,  up to two admits for Theorem \ref{t;inverse} 
of \cite{CSP_7}   and Theorem \ref{thm:dimitrov} from \cite{VD}. 
It can be inspected at  \href{https://github.com/lacey-gatech/restricted-polynomials-lean/}{this repository}.

\section{Notation and Background}\label{sec:notation}


We proceed to record standard concepts from the analysis of Boolean functions, which we will make use of repeatedly. We refer the interested reader to \cite{ODonnell} for a comprehensive exposition of the subject.

\emph{Conventions.} The letter $\mu$ is reserved for the joint distribution on a product $\Sigma_1\times\cdots\times\Sigma_m$ of alphabets, such as the distribution detecting our configurations. A probability measure on a \emph{single} alphabet $\Sigma$ is denoted by $\pi$. In our application every $\Sigma_j$ is $\F_p$, and every marginal of $\mu$ is the uniform measure on $\F_p$ (Lemma \ref{l;large exp}), so $\pi$ will be the uniform measure. The letter $\ap$ always denotes the density of a set, and $T$ always denotes the translation operator $Tf(x)=f(x+1)$.

In what follows, $\Sigma$ is a finite set equipped with a probability measure $\pi$. A variable $x$ chosen according to $\pi$ is denoted $x\sim \pi$, and we write $x\sim \Sigma$ if $x$ is chosen uniformly from $\Sigma$. Expectations of functions \( f \colon \Sigma \to    \mathbb{C} \) are given by
\begin{equation}
    \mathbb{E}_\pi f = \sum_{x\in \Sigma} f(x) \pi(x).
\end{equation}

\begin{definition}
    For $0< \rho < 1$, the \textit{$\rho$-correlated distribution} on $(\Sigma, \pi)$ is defined as follows. For $x\in \Sigma$, we write $y \sim_{\rho} x$ to indicate that $y$ is randomly chosen as follows: with probability $\rho$ we set $y=x$, and with probability $1- \rho$ we draw $y \sim \pi$ independently of $x$. Given a function $f \colon \Sigma \to \mathbb{C}$, we set $N_ \rho f (x) \coloneqq \mathbb{E}_{y\sim_{\rho} x} f(y)$, the \textit{\( \rho   \)-noise operator}. On $\Sigma^n$ with the product measure $\pi^{\otimes n}$, we use the same symbol $N_\rho$ for the tensor product operator $N_\rho^{\otimes n}$, which applies the noise independently in each coordinate.
\end{definition}

\begin{definition}
    Let $0\leq \rho <1$. The \textit{$\rho$-stability} of $f \in L^2( \Sigma^n, \pi^{\otimes n}) $ is given by
    \begin{equation}
    \operatorname{Stab}_\rho(f) \coloneqq \langle f, N_ \rho f \rangle_{L^2(\pi^{\otimes n})}.
\end{equation}
    When $f$ is defined on $\Sigma_j^n$ for one of the alphabets of a joint distribution $\mu$, stability is always taken with respect to the marginal $\pi = \mu_j$ of $\mu$ on the $j$-th coordinate. This is the convention in Theorem \ref{t;inverse}.
\end{definition}

Abelian embeddings were defined in Definition \ref{d;abelianEmbedding}. We say that an abelian embedding is \textit{trivial} if all of its maps are constant, and that $\mu$ \textit{admits no nontrivial abelian embedding} if every abelian embedding of $\supp\mu$ into every abelian group is trivial.

We take the $n$-fold product $\Sigma^{n}$ with its product measure $\pi^{\otimes n}$. For a subset of indices $S\subseteq [n]$, we write $\Sigma ^S = \prod_{s\in S} \Sigma$ and $S^{c} \coloneqq [n] \setminus S$ for its complement; complements are always denoted by a superscript $c$.

\begin{definition}
  The \textit{Efron-Stein decomposition} of a function $f \colon \Sigma^n \to \mathbb{C}$
  with respect to $\pi^{\otimes n}$ is defined in terms of a decomposition of \( L^2(\Sigma^n, \pi^{\otimes n}) \).
  For \( S \subseteq [n] \), set  \( V_S \) to be all functions \( f \) which are of the form
  \( f(x) = g(x|_S) \), for some function \( g \colon \Sigma^S \to \mathbb{C} \).
  Then, define \(  V _{\subseteq S}  \) to be the span of \( \{ V_{S'} \colon S'\subseteq S\}\).
  The subspace \( V_{=S} = V_{\subseteq S} \cap \bigcap_{T'\subsetneqq S} V_{\subseteq T'} ^{\perp}   \).
  It can be shown that the spaces \( V_{=S} \) are pairwise orthogonal, and hence
  \( L^2(\Sigma^n)= \bigoplus _{S\subseteq [n]} V _{=S}   \).
  The orthogonal projection of \( f \) onto \( V_{=S}  \) is given by
 \begin{align}
    f^{=S}(x_S)  = \sum_{S'\subseteq S} (-1) ^{\lvert S\setminus S'\rvert} \mathbb{E}_{\pi^{\otimes (S')^{c}}} (f (y) \mid y|_{S'}=x_{S'}).
 \end{align}
 Here the conditional expectation fixes the coordinates in $S'$ and averages $f$ over the complementary coordinates $y|_{(S')^{c}}\sim \pi^{\otimes (S')^{c}}$.
 The \textit{Efron-Stein decomposition} of $f$ is then given by
 \begin{align*}
    f (x) = \sum_{\substack{S \subseteq [n] }} f^{=S} (x_S).
\end{align*}
In particular, $f^{=\emptyset}= \mathbb{E}f $.
\end{definition}

\begin{definition}
    Let $\Sigma$ be a finite set, $n \geq 1$, and let $\pi$ be a probability distribution on $\Sigma$. A function $f: \Sigma^n  \rightarrow \C$ is \textit{homogeneous of degree $d$} if $f = \sum_{S; |S| = d} f^{=S}$. In addition, the \textit{weight of $f$ at degree at most $d$} is defined by
\begin{equation}
    W_{\leq d} (f) \coloneqq \sum_{\lvert S\rvert\leq d} \lVert f^{=S} \rVert_2^2.
\end{equation}
\end{definition}

The weight of a function can also be described using a Fourier basis. Let $s=|\Sigma|$, and fix an orthonormal basis $\phi_0\equiv 1,\phi_1,\ldots,\phi_{s-1}$ of $L^2(\Sigma,\pi)$ whose first element is the constant function. For a \textit{multi-index} $\xi \in \{0,1,\ldots,s-1\}^n$, set $\phi_{\xi}(x)\coloneqq\prod_{i=1}^n \phi_{\xi_i}(x_i)$. These functions form an orthonormal basis of $L^2(\Sigma^n,\pi^{\otimes n})$. We write $\supp \xi =\{i: \ \xi_i \neq 0\}$ and $\# \xi = |\supp \xi|$. Then
\begin{align*}
    f^{=S} = \sum_{\substack{\xi:\ \supp \xi = S}} \wh{f}(\xi) \phi_{\xi}, \qquad \wh{f}(\xi) = \la f, \phi_{\xi} \ra .
\end{align*}
By Plancherel,
\begin{align}
    W_{\leq d} (f) = \sum_{\substack{\xi:\ \# \xi \leq d }} |\wh{f}(\xi)|^2.
\end{align}
In addition, the $\rho$-noise operator is given by \( N_ \rho f = \sum_ {S \subseteq [n]} \rho ^{\lvert S\rvert}  f^{=S}\).
This follows since $f^{=S}$ depends only on the coordinates indexed by $S$ and satisfies \(\E_i  f^{=S} =0 \) for every $i \in S$.

\section{Reduction to cyclic groups of prime order}\label{sec:reduction} 

 In order to apply the inverse theorem (Theorem \ref{t;inverse}) in our context, one must rule out abelian embeddings for the distribution $\mu$ on $\mathbb{F}_p^{\ell + 1}$ supported on points of the form
        \begin{equation} \label{restricted progression in one coordinate} (x, x+y, x+y^2,\ldots,x+y^{\ell})\end{equation}
    where $x \in \F_p$ and $y \in Y = \{0,2,\ldots,M\}$ for some $M < p$.\footnote{For the $k$-term progressions of Theorem \ref{main thm}, the largest power is $\ell = k-1$. Writing $\ell$ rather than $k$ for configurations of the form in formula \eqref{restricted progression in one coordinate} keeps the letter $k$ reserved for the number of terms.} The goal of Sections \ref{sec:reduction}--\ref{sec:diffChar} is the proof of the following proposition, which gives bounds on the sizes of $M$ and $p$ (depending on $\ell$) forbidding abelian embeddings of $\mu$: 

    \begin{proposition}\label{prop:noAbelianEmbeddings}
    Let $\ell\geq 1$, and let $M(\ell)=\exp(O(\ell\log\ell))$ and $P(\ell)=\exp\!\big(\exp(O(\ell^2\log\ell))\big)$ be the functions provided by Proposition \ref{prop:uniformTorsionRank}. Let $M \geq \max\{M(\ell), \ell+1\}$ be an integer and let $p > \max\{M, P(\ell)\}$ be prime. Let $\mu$ be a probability distribution on $\F_p^{\ell+1}$ whose support is exactly the set 
    \begin{align*}
        \Big\{(x,x+y,x+y^2,\hdots,x+y^\ell) : \ x\in \mathbb{F}_p,\ y\in \{0,2,\ldots, M\}\Big\}.
    \end{align*}
    Then $\mu$ admits no nontrivial abelian embedding.
\end{proposition}

\begin{remark} 
    The proofs below use every value $y \in \{0,2,\ldots,M\}$: the value $y=0$ eliminates $f_0$, and the values $y \in \{2,\ldots,M\}$ supply the rows of the matrices $A_{\ell,M}$ that appear below.  The distribution $\mu$ of Lemma \ref{l;large exp} gives positive mass to every point of this set.  
\end{remark}

The next remark shows that the number of steps cannot be reduced: the hypothesis $M\geq \ell+1$ of Proposition \ref{prop:noAbelianEmbeddings} is necessary, and not an artifact of our argument.

\begin{remark}[Sharpness of $M \geq \ell+1$]\label{r;sharpM}
    Let $\ell \geq 1$, let $p$ be prime, and let $Y = \{0\}\cup Y'$ with $Y' \subseteq \F_p\setminus\{0\}$. Let $\mu$ be any distribution whose support is 
    \begin{align*}
        \Big\{(x,x+y,x+y^2,\hdots,x+y^\ell) : \ x\in \mathbb{F}_p,\ y\in Y\Big\}.
    \end{align*}
    If $\lvert Y'\rvert<\ell$, then $\mu$ admits a nontrivial abelian embedding into $(\F_p,+)$.

    Indeed, the linear system $\sum_{j=1}^\ell c_j y^j = 0$, $y \in Y'$, consists of fewer than $\ell$ equations in the $\ell$ unknowns $c_1,\ldots,c_\ell$, so it has a solution with some $c_j \neq 0$. Put
    \begin{align*}
        f_j(x) \coloneqq c_j x \quad (1\leq j \leq \ell), \qquad f_0(x) \coloneqq -\Big(\sum_{j=1}^\ell c_j\Big) x .
    \end{align*}
    For every $x \in \F_p$ and every $y \in Y$,
    \begin{align*}
        f_0(x) + \sum_{j=1}^\ell f_j(x+y^j) = -\Big(\sum_{j=1}^\ell c_j\Big)x + \sum_{j=1}^\ell c_j x + \sum_{j=1}^\ell c_j y^j = \sum_{j=1}^\ell c_j y^j,
    \end{align*}
    which vanishes for $y\in Y'$ by the choice of the $c_j$, and also for $y = 0$. The map $f_j(x) = c_jx$ with $c_j \neq 0$ is nonconstant, so the embedding is nontrivial.

    For the step sets $Y = \{0,2,\ldots,M\}$ of Proposition \ref{prop:noAbelianEmbeddings} we have $\lvert Y'\rvert = M-1$, so $M \geq \ell+1$ is necessary. By Lemma \ref{lem:sameChar} it is also sufficient, when $p>M$, to rule out embeddings into $(\F_p,+)$, so for the equal-characteristic case $M\geq \ell+1$ is exactly the right threshold. The larger value $M(\ell)$ is needed only to handle the remaining characteristics $q \neq p$.
\end{remark}

    In the remainder of this section we show that Proposition \ref{prop:noAbelianEmbeddings} reduces to the following two lemmas:

    \begin{lemma}\label{lem:sameChar}Let $\ell \geq 1$, let $M\geq \ell+1$, and let $p > M$ be prime. If $f_0,\hdots,f_\ell:\mathbb{F}_p\rightarrow \F_p$ satisfy
    \begin{align*}
        f_0 (x) + \sum_{j=1}^{\ell}f_j(x+y^j)=0 \qquad \text{for all } x \in \F_p \text{ and } y\in \{0,2,\hdots,M\},
    \end{align*}
    then $f_j$ is constant for every $j = 0,1, \ldots, \ell$. That is, $\mu$ admits no nontrivial abelian embedding into $(\F_p,+)$. \end{lemma}

    \begin{lemma} \label{lem:diffChar}
    Let $\ell \geq 1$. There exist $M(\ell), P(\ell) \in \N$ for which the following holds. Let $M\geq M(\ell)$, let $p>P(\ell)$ be prime, and let $q\neq p$ be prime. If $f_0,\ldots,f_\ell:\F_p\to\F_q$ satisfy 
        \begin{align}\notag
        f_0 (x) + \sum_{j=1}^{\ell}f_j(x+y^j)=0 \qquad \text{for all } x \in \F_p \text{ and } y\in \{0,2,\hdots,M\},
    \end{align}
    then $f_j$ is constant for every $j=0,1,\ldots,\ell$. That is, $\mu$ admits no nontrivial abelian embedding into $(\F_q,+)$. Quantitatively, one can take $M(\ell)=\exp(O(\ell\log\ell))$ and $P(\ell)=\exp\!\big(\exp(O(\ell^2\log\ell))\big)$.
    \end{lemma}

    The proof of Lemma \ref{lem:sameChar} will be the subject of Section \ref{sec:sameChar} and the proof of Lemma \ref{lem:diffChar} will be the subject of Section \ref{sec:diffChar}. 

    \begin{remark}
        Section \ref{sec:diffChar} develops two approaches to Lemma \ref{lem:diffChar}, with the same initial reduction to Proposition \ref{prop:uniformTorsionRank}. A qualitative argument establishes the existence of $M(\ell)$ and $P(\ell)$ while the second obtains the quantitative bounds stated above. 
    \end{remark}

The following lemma shows that one can reduce abelian embeddings to the case of cyclic groups of prime order, which is what allows us to reduce Proposition \ref{prop:noAbelianEmbeddings} to Lemmas \ref{lem:sameChar} and \ref{lem:diffChar}. 

\begin{lemma}\label{lem:cyclicReduction}
    Suppose $\mu$ is a probability distribution on a finite set $\Sigma_0  \times \Sigma_1 \times \cdots \times \Sigma_\ell$ which admits a nontrivial abelian embedding into some abelian group. Then there exists a prime $q$ for which $\mu$ admits a nontrivial abelian embedding into the cyclic group $(\Z/q\Z,+)$.
\end{lemma}

To avoid a clash with the prime $p$ of $\F_p^n$, we write $q$ for the prime produced by Lemma \ref{lem:cyclicReduction}. In our application $\Sigma_j = \F_p$ for every $j$, and $q$ may or may not equal $p$.

\begin{proof}
    Let $(G,+)$ be an abelian group and suppose the maps $f_i:\Sigma_i \rightarrow G$, $0 \leq i \leq \ell$, form a nontrivial abelian embedding of $\mu$ into $G$. The sets $\Sigma_i$ are finite, so the maps $f_0,\ldots,f_\ell$ have finite images. Replacing $G$ by the subgroup generated by these images, we may assume that $G$ is finitely generated. Then there exist nonnegative integers $r, n_1,\ldots,n_s$ and primes $q_1,\ldots,q_s$ such that $G \cong \Z^r \oplus \Z/q_1^{n_1} \Z \oplus \cdots \oplus  \Z/q_s^{n_s}\Z$. Composing $f_0,\ldots,f_\ell$ with the projection onto any cyclic summand gives an abelian embedding into that summand, since projections are homomorphisms. At least one of these embeddings is nontrivial: if all were trivial, then every coordinate projection of every $f_i$ would be constant, and hence every $f_i$ would be constant, since an element of $G$ is determined by its coordinates. There are two cases.

    \emph{Case 1: a nontrivial embedding $f_0',\ldots,f_\ell'$ into $\Z$.} Some $f_i'$ takes two distinct values. Let $q$ be a prime larger than $|f_i'(a) - f_i'(b)|$ for all $0 \leq i \leq \ell$ and $a, b \in \Sigma_i$. Let $\pi:\Z \rightarrow \Z/q\Z$ be reduction mod $q$. Then $\pi \circ f_0',\ldots,\pi \circ f_\ell'$ is an abelian embedding into $\Z/q\Z$, and it is nontrivial because $\pi$ is injective on the finite set of values of each $f_i'$.

    \emph{Case 2: a nontrivial embedding into $\Z/q^n \Z$ for some prime $q$ and $n\geq 1$.} We show by induction on $n$ that there is then a nontrivial embedding into $\Z/q\Z$. The case $n=1$ is immediate. Suppose the claim holds for some $n\geq 1$, and let $f_0,\ldots,f_\ell$ be a nontrivial embedding into $\Z/q^{n+1}\Z$. We identify $\Z/q^{n+1}\Z$ with the residues $\{0,1,\ldots,q^{n+1}-1\}$. The maps $g_j \coloneqq f_j \bmod q^n$ form an embedding into $\Z/q^n\Z$, since reduction mod $q^n$ is a homomorphism. If this embedding is nontrivial, the induction hypothesis finishes the proof.

    Otherwise, every $g_j$ is constant, say $g_j \equiv c_j$ with $c_j \in \{0,1,\ldots,q^n-1\}$. Then, write each \( f_j \) as 
    \begin{align*}
        f_j(x) = c_j + q^n\phi_j(x), \qquad \phi_j:\Sigma_j \to \{0,1,\ldots,q-1\},
    \end{align*}
    and since the $f_j$ are not all constant, the $\phi_j$ are not all constant. Pick any $(x_0,\ldots,x_\ell) \in \supp\mu$. Reducing $\sum_j f_j(x_j) \equiv 0 \pmod{q^{n+1}}$ modulo $q^n$ gives $\sum_j c_j \equiv 0 \pmod{q^n}$. Write $\sum_j c_j = q^n c$ with $c \in \Z$. Then for every $(x_0,\ldots,x_\ell)\in\supp\mu$,
    \begin{align*}
        0 \equiv \sum_{j=0}^\ell f_j(x_j) = q^n\Big(c + \sum_{j=0}^\ell \phi_j(x_j)\Big) \pmod{q^{n+1}},
    \end{align*}
    so $c + \sum_{j} \phi_j(x_j) \equiv 0 \pmod q$. Hence the maps $\phi_0 + c,\ \phi_1,\ldots,\phi_\ell$, reduced mod $q$, form an abelian embedding into $\Z/q\Z$. It is nontrivial because each $\phi_j$ takes values in $\{0,\ldots,q-1\}$, where reduction mod $q$ is injective, and some $\phi_j$ is nonconstant.
\end{proof}

\begin{proof}[Proof of Proposition \ref{prop:noAbelianEmbeddings} assuming Lemmas \ref{lem:sameChar} and \ref{lem:diffChar}]  
     Suppose that $\mu$ admits a nontrivial abelian embedding. By Lemma \ref{lem:cyclicReduction}, it then admits a nontrivial abelian embedding $f_0,\ldots,f_\ell:\F_p\to\mathbb{Z}/q\mathbb{Z}$ for some prime $q$. Since the support of $\mu$ is the whole set in Proposition \ref{prop:noAbelianEmbeddings}, these maps satisfy $f_0(x)+\sum_{j=1}^\ell f_j(x+y^j)=0$ for all $x\in\F_p$ and all $y\in\{0,2,\ldots,M\}$. If $q=p$, Lemma \ref{lem:sameChar} (which applies since $M\geq \ell+1$ and $p>M$) shows that every $f_j$ is constant. If $q\neq p$, Lemma \ref{lem:diffChar} (which applies since $M\geq M(\ell)$ and $p>P(\ell)$) shows the same. Either way the embedding is trivial, a contradiction.
\end{proof}

\section{No abelian embeddings into finite fields with the same characteristic}\label{sec:sameChar} The goal of this section is to prove Lemma \ref{lem:sameChar}, which we restate here:

\begin{sameCharRestatement}
    Let $\ell \geq 1$, let $M\geq \ell+1$, and let $p > M$ be prime. If $f_0,\hdots,f_\ell:\mathbb{F}_p\rightarrow \F_p$ satisfy
    \begin{align}\label{eq:sameCharSystem}
        f_0 (x) + \sum_{j=1}^{\ell}f_j(x+y^j)=0 \qquad \text{for all } x \in \F_p \text{ and } y\in \{0,2,\hdots,M\},
    \end{align}
    then $f_j$ is constant for every $j = 0,1, \ldots, \ell$. That is, $\mu$ admits no nontrivial abelian embedding into $(\F_p,+)$.
\end{sameCharRestatement}

\begin{proof}
    Evaluating \eqref{eq:sameCharSystem} at $y=0$ gives
    \begin{align}\label{eq:yZero}
        f_0(x)+f_1(x)+\hdots+f_\ell(x)=0, \qquad x \in \F_p.
    \end{align}
    Subtracting \eqref{eq:yZero} from \eqref{eq:sameCharSystem} eliminates $f_0$:
    \begin{align}\label{eq:sameCharDiff}
        \sum_{j=1}^\ell (T^{y^j}-\mathrm{Id})f_j=0, \qquad y\in \{2,\hdots,M\},
    \end{align}
    where $T$ is the translation operator $Tf(x) \coloneqq f(x+1)$ on functions $\F_p \to \F_p$ and $\mathrm{Id}$ is the identity operator.

    \emph{The ring $R$.} Since $T^p = \mathrm{Id}$, the space of functions $\F_p \to \F_p$ is a module over $R\coloneqq\mathbb{F}_p[T]/(T^p-1)$, in which $1\in R$ acts as $\mathrm{Id}$; inside $R$ we therefore write $1$ for $\mathrm{Id}$. In characteristic $p$ we have $T^p-1=(T-1)^p$, so $T - 1$ is nilpotent in $R$. Consequently, an element $F(T) \in R$ is a unit if and only if $F(1)\neq 0$. Indeed, write $F(T) = F(1) + (T-1)G(T)$. If $F(1) \neq 0$, then $F(T) = F(1)\big(1 + N\big)$ with $N = F(1)^{-1}(T-1)G(T)$ nilpotent, and $1+N$ has inverse $\sum_{i=0}^{p-1}(-N)^i$. If $F(1) = 0$, then $F(T)$ is a multiple of the nilpotent element $T-1$, hence not a unit.

    \emph{Factoring out $T - 1$.} For $y\geq 2$ and $1\leq j\leq \ell$,
    \begin{align*}
        T^{y^j}-1 = (T-1)B_{y,j}(T), \qquad B_{y,j}(T) \coloneqq 1+T+\cdots + T^{y^j-1}, \qquad B_{y,j}(1) = y^j.
    \end{align*}
    Setting $g_j \coloneqq (T-\mathrm{Id})f_j$, and using that polynomials in $T$ commute, \eqref{eq:sameCharDiff} becomes $\sum_{j=1}^\ell B_{y,j}(T)\,g_j = 0$ for $y\in\{2,\ldots,M\}$. Keep only the $\ell$ rows $y_r \coloneqq r+1$, $1 \leq r \leq \ell$. These lie in $\{2,\ldots,M\}$ because $M\geq \ell+1$. Let $B(T)\coloneqq \big(B_{y_r,j}(T)\big)_{1\leq r,j\leq \ell}$, an $\ell\times \ell$ matrix over $R$. Then
    \begin{align*}
        B(T)\,\vec g = 0, \qquad \vec g \coloneqq (g_1,\ldots,g_\ell)^{\top}.
    \end{align*}

    \emph{Invertibility.} The determinant commutes with the ring homomorphism $R \to \F_p$, $T \mapsto 1$, so $\det B(T)$ evaluated at $T=1$ equals $\det\big(y_r^j\big)_{1\leq r,j\leq \ell}$. Factoring $y_r$ out of the $r$-th row leaves the Vandermonde matrix $\big(y_r^{j-1}\big)_{1\leq r,j\leq \ell}$, so, since $y_r=r+1$,
    \begin{align*}
        \det\big(y_r^j\big)_{1\leq r,j\leq \ell}=\prod_{r=1}^\ell y_r\cdot \prod_{1\leq r<s\leq \ell}(y_s-y_r)=\prod_{r=1}^{\ell}(r+1)\cdot\prod_{1\leq r<s\leq \ell}(s-r).
    \end{align*}
    Each factor is an integer in $\{1,\ldots,M\}$, hence nonzero modulo $p$ because $p>M$; as $p$ is prime, the product of these factors is nonzero in $\F_p$ as well. Thus $\det B(T)$ is a unit in $R$.
    In a ring, we have the identity using the adjugate of \( B(T) \): 
    \( \textup{adj}(B(T))B(T)=\textup{det}(B)I \). 
    And, hence,  $\det B(T)\, g_j = 0$, and therefore $g_j = 0$ for each $j$.

    \emph{Conclusion.} For $1 \leq j \leq \ell$, $g_j = 0$ means $f_j(x+1)=f_j(x)$ for all $x\in\F_p$, so $f_j$ is constant. By \eqref{eq:yZero}, $f_0=-(f_1+\cdots+f_\ell)$ is constant as well.
\end{proof}
 
\section{No abelian embeddings into finite fields of different characteristic}\label{sec:diffChar}

This section is dedicated to the proof of Lemma \ref{lem:diffChar}, which we restate here:

\begin{diffCharRestatement}
    Let $\ell \geq 1$. There exist $M(\ell), P(\ell) \in \N$ for which the following holds. Let $M\geq M(\ell)$, let $p>P(\ell)$ be prime, and let $q\neq p$ be prime. If $f_0,\ldots,f_\ell:\F_p\to\F_q$ satisfy 
        \begin{align}\notag
        f_0 (x) + \sum_{j=1}^{\ell}f_j(x+y^j)=0 \qquad \text{for all } x \in \F_p \text{ and } y\in \{0,2,\hdots,M\},
    \end{align}
    then $f_j$ is constant for every $j=0,1,\ldots,\ell$. That is, $\mu$ admits no nontrivial abelian embedding into $(\F_q,+)$. Quantitatively, one can take $M(\ell)=\exp(O(\ell\log\ell))$ and $P(\ell)=\exp\!\big(\exp(O(\ell^2\log\ell))\big)$.
    \end{diffCharRestatement}

We reduce Lemma \ref{lem:diffChar} to the following proposition, for which we give two proofs, one qualitative and the other quantitative:

\begin{proposition}\label{prop:uniformTorsionRank}
    For integers $\ell\geq 1$ and $M\geq 2$ define the $(M-1)\times \ell$ matrix \begin{align*}
        A_{\ell,M}(z)\coloneqq(z^{y^j}-1)_{\substack{2\leq y\leq M\\ 1\leq j\leq \ell}}.
    \end{align*} Then there exist integers $M=M(\ell)$ and $P=P(\ell)$, depending only on $\ell$, with the following property. If $p>P$ and $q\neq p$ are primes, and $\lambda \in \overline{\mathbb{F}}_q^\times$ has multiplicative order $p$, then $A_{\ell,M}(\lambda)$ has rank $\ell$. Quantitatively, one can take $M(\ell)=\exp(O(\ell\log\ell))$ and $P(\ell)=\exp\!\big(\exp(O(\ell^2\log\ell))\big)$.
\end{proposition}

We first reduce Lemma \ref{lem:diffChar} to Proposition \ref{prop:uniformTorsionRank} in Subsection \ref{sec:rankReduction}, and then give two proofs of the proposition, one qualitative and the other quantitative, in subsections \ref{sec:qualitativeLem} and \ref{sec:quantitativeLem} respectively. The qualitative proof in Subsection \ref{sec:qualitativeLem} establishes the existence of \(M(\ell)\) and \(P(\ell)\), while the quantitative proof in Subsection \ref{sec:quantitativeLem} gives explicit bounds for these parameters.

Both proofs proceed similarly. We first find a finite cutoff \(M(\ell)\) such that any nonzero algebraic number $\theta$ at which \(A_{\ell,M(\ell)}(\theta)\) has deficient rank over $\overline{\mathbb{Q}}$ must be a root of unity (this corresponds to Lemma \ref{lem:fintiteDefRankImpliesUnity} in the qualitative approach and Lemma \ref{thm:quantitativeRankDeficiency} in the quantitative approach). We then use this conclusion to bound, uniformly over finite characteristics \(q\), the possible prime orders of elements $\lambda \in \overline{\mathbb{F}}_q^{\times}$ at which the matrix $A_{\ell,M}(\lambda)$ has deficient rank. Choosing \(P(\ell)\) larger than these possible orders proves Proposition \ref{prop:uniformTorsionRank}. The two approaches differ in how they obtain the cutoff and control the passage to finite characteristic.

\subsection{Reduction to Proposition \ref{prop:uniformTorsionRank}}\label{sec:rankReduction}

We first reduce Lemma \ref{lem:diffChar} to Proposition \ref{prop:uniformTorsionRank}. Fix distinct primes $p,q$ and maps $f_0,\ldots,f_\ell:\F_p\to\F_q$ satisfying the hypothesis of Lemma \ref{lem:diffChar}. The bounds on $M$ and $p$ will enter when we apply Proposition \ref{prop:uniformTorsionRank}. As in the previous section, consider the translation operator $Tf(x)\coloneqq f(x+1)$. The hypothesis of the lemma becomes
\begin{align}\label{eq no ab restriction}
    f_0(x) + \sum_{j=1}^\ell T^{y^j} f_j(x) = 0, \qquad x\in\F_p,\quad y\in\{0,2,3,\ldots,M\}.
\end{align}
Observe that $(T^p - \mathrm{Id})f =0$ for every function $f: \F_p \rightarrow \F_q$. 
Since $p \neq q$, the polynomial $z^p -1$ splits into distinct linear factors over the algebraic closure\footnote{That is, $\overline{\F}_q$ denotes an algebraic closure containing $\F_q$ such that every nonconstant polynomial with coefficients in $\F_q$ splits completely into linear factors over $\overline{\F}_q$.} $\overline{\F}_q$. As such, the translation operator $T$ is diagonalizable over $\overline{\F}_q$. We extend scalars: each $f_j$ is regarded as a function $\F_p\to\overline{\F}_q$, and $T$ acts on the $\overline{\F}_q$-vector space of all such functions. We decompose every $f_j$ with respect to the eigenspace decomposition of $T$ on this space and write 
\begin{align*}
    f_j = \sum_{\lambda^p =1} f_{j,\lambda}, \hspace{.5in} \textrm{where } T f_{j,\lambda}  = \lambda f_{j,\lambda}. 
\end{align*}
For every $p$-th root of unity $\lambda \in \overline{\F}_q$ (including $\lambda=1$), the $\lambda$-eigenspace is one-dimensional over $\overline{\F}_q$. In addition, an eigenfunction with eigenvalue $\lambda$ has, up to a scalar multiple, the form $g_{\lambda}(x) = \lambda^x$. Here $x\in\F_p$ is represented by an integer; since $\lambda^p=1$, the value $\lambda^x$ does not depend on the choice of representative. To verify this assertion, by repeatedly applying the translation operator to an eigenfunction $g_{\lambda}$, we obtain
\begin{align*}
    g_{\lambda}(x +1) = T g_{\lambda} (x) = \lambda g_{\lambda} (x), \hspace{.2in} \forall x \in \F_p 
\end{align*}
Thus, $g_{\lambda}(x) = \lambda^x g_{\lambda}(0)$, where we henceforth normalize the eigenfunctions to be $\lambda^x$. In addition, $\lambda^p =1$. The eigenspace decomposition for each function is thus given by
\begin{align*}
    f_j (x) = \sum_{\lambda^p =1} f_{j,\lambda} (x) = \sum_{\lambda^p =1} f_{j, \lambda}(0) \lambda^x.
\end{align*}
The eigenspace for $\lambda=1$ consists exactly of the constant functions. Thus, to prove Lemma \ref{lem:diffChar}, it suffices to show that $f_{j,\lambda}(0)=0$ whenever $\lambda\neq 1$. We now translate \eqref{eq no ab restriction} into equations for these coefficients:
\begin{align*}
     \sum_{\lambda^p =1 }f_{0, \lambda}(x)+\sum_{j=1}^\ell T^{y^j}  \sum_{\lambda^p = 1} f_{j, \lambda} (x)  =    \sum_{\lambda^p = 1} \lp f_{0, \lambda}(x)+ \sum_{j=1}^\ell \lambda^{ y^j} f_{j, \lambda} (x) \rp  = 0 
\end{align*}
By further expanding, the equation above can be rewritten as
\begin{align*}
    \sum_{\lambda^p = 1} \lambda^x \lp f_{0, \lambda}(0) + \sum_{j=1}^\ell \lambda^{ y^j} f_{j, \lambda} (0) \rp  = 0.
\end{align*}
The functions $x \mapsto \lambda^x$, $\lambda^p = 1$, are eigenfunctions of $T$ with distinct eigenvalues, hence linearly independent over $\overline{\F}_q$. Hence, for each $p$th root of unity $\lambda$, 
\begin{align*}
    f_{0, \lambda}(0) + \sum_{j=1}^\ell \lambda^{ y^j} f_{j, \lambda} (0) = 0 
\end{align*}
where $y \in \{0, 2, 3, \ldots, M\}$. The equation with $y=0$ determines $f_{0,\lambda}(0)$ from the remaining coefficients. Subtracting it eliminates $f_{0,\lambda}(0)$ and gives, for $y\in\{2,3,\ldots,M\}$,
\begin{align*}
    \sum_{j=1}^\ell  (\lambda^{ y^j} -1) f_{j, \lambda} (0) = 0 
\end{align*}
For every fixed $\lambda$, we obtain $M-1$ equations in the $\ell$ coefficients $f_{1,\lambda}(0),\ldots,f_{\ell,\lambda}(0)$, of the form
\begin{align*}
    \begin{pmatrix}
        \lambda^y -1 & \lambda^{y^2} -1 & \cdots & \lambda^{y^\ell} -1
    \end{pmatrix} \begin{pmatrix}
        f_{1, \lambda}(0)\\
        f_{2, \lambda}(0)\\
        \vdots\\
        f_{\ell, \lambda}(0)
    \end{pmatrix} = 0 . 
\end{align*}
Regrouping the $M-1$ equations, we obtain the matrix equation
\begin{equation}
    \begin{split}
        A_{\ell, M} (\lambda) \vec{f} = \begin{pmatrix}
        \lambda^2 -1 & \lambda^{2^2} -1 & \cdots & \lambda^{2^\ell}-1\\
        \lambda^3 -1 & \lambda^{3^2} -1 & \cdots & \lambda^{3^\ell} -1 \\
        \lambda^4 -1  & \lambda^{4^2} -1 & \cdots & \lambda^{4^\ell} -1\\
        \vdots & \vdots & & \vdots \\
        \lambda^{M}-1 & \lambda^{M^2}-1 & \cdots & \lambda^{M^\ell}-1
    \end{pmatrix} \begin{pmatrix}
        f_{1, \lambda}(0)\\
        f_{2, \lambda}(0)\\
        \vdots\\
        f_{\ell, \lambda}(0)
    \end{pmatrix} = 0 
    \end{split}
\end{equation}
If $A_{\ell,M}(\lambda)$ has full column rank $\ell$, this matrix equation forces $f_{j,\lambda}(0)=0$ for $1\leq j\leq\ell$. Consequently, if full rank holds for every $\lambda\neq 1$ with $\lambda^p=1$, then $f_1,\ldots,f_\ell$ have only their $\lambda=1$ components and are constant. The case $y=0$ of \eqref{eq no ab restriction} then gives $f_0=-\sum_{j=1}^\ell f_j$, so $f_0$ is constant as well. This is precisely the conclusion of Lemma \ref{lem:diffChar}. 

The same conclusion holds for every larger step cutoff: if $M\geq M(\ell)$, the matrix $A_{\ell,M}(\lambda)$ contains $A_{\ell,M(\ell)}(\lambda)$ as its first $M(\ell)-1$ rows. Thus any choices of $M(\ell)$ and $P(\ell)$ in Proposition \ref{prop:uniformTorsionRank} give the corresponding choices in Lemma \ref{lem:diffChar}. This shows how Lemma \ref{lem:diffChar} reduces to Proposition \ref{prop:uniformTorsionRank}.

\subsection{Qualitative proof of Proposition \ref{prop:uniformTorsionRank}}\label{sec:qualitativeLem}

As mentioned earlier in this section, the first step in proving Proposition \ref{prop:uniformTorsionRank} is finding some $M$ such that any $\theta$ causing $A_{\ell,M}(\theta)$ to have rank $< \ell$ over $\overline{\mathbb{Q}}$ must be zero or a root of unity (this will be the content of Lemma \ref{lem:fintiteDefRankImpliesUnity}). For our qualitative approach, we do this by letting $M$ tend to infinity, and considering the infinite matrix analogue $A_{\ell,\infty} ( \theta)$. After importing Kronecker's theorem, this allows us to identify that if a nonzero algebraic number $\theta$ satisfies that $A_{\ell,\infty}(\theta)$ has rank strictly less than $\ell$, then $\theta$ is a root of unity (Lemma \ref{lem:defRankImpliesUnity}). 
We then use a Noetherian (i.e. compactness) argument to pass to finitely many rows (values of $y$), which is the main reason why this argument is existential (Lemma \ref{lem:finitelyManyRowsSuffice}). 
The greatest common divisor $D(z)$ of the minors of $A_{\ell,M}(z)$ has finitely many roots, all roots of unity or $0$, so only finitely many cyclotomic orders can obstruct full rank.
 Reducing modulo $q$ then works for all but finitely many exceptional primes $q$ (Lemma \ref{lem:goodCharacteristics}), which are handled with a degree bound (Lemma \ref{lem:exceptionalCharacteristics}).

For the argument, one will need to recall a few facts from algebraic number theory.

Recall that a \emph{place} of a number field $K$ is an equivalence class of nontrivial absolute values on $K$. Up to equivalence, the places of $K$ are the archimedean ones, $|x|_v = |\iota(x)|$ for a complex embedding $\iota:K\to\C$, and the non-archimedean ones, induced by the prime ideals of the ring of integers $\mathcal{O}_K$.

\begin{lemma}[An implication of Kronecker's theorem]\label{lem:kronecker} Let $K$ be a number field and let $\theta \in K^\times$ be not a root of unity. Then there exists a place $v$ of $K$ such that $|\theta|_v\neq 1$.
\end{lemma}

\begin{proof}
    Suppose that $|\theta|_v = 1$ for every place $v$ of $K$. The non-archimedean places show that $\theta$ lies in every localization of $\mathcal{O}_K$ at a prime, so $\theta\in\mathcal{O}_K$ is an algebraic integer. The archimedean places show that every Galois conjugate of $\theta$ has absolute value $1$, since each conjugate is $\iota(\theta)$ for some embedding $\iota$. By Kronecker's theorem, an algebraic integer all of whose conjugates lie on the unit circle is a root of unity, a contradiction.
\end{proof}

\begin{lemma}[An implication of Hilbert's basis theorem] If $K$ is a Noetherian ring (all ideals are finitely generated), then the polynomial ring $K[z]$ is also Noetherian. In particular, $\mathbb{Q}[z]$ is a Noetherian ring.
    
\end{lemma}
  
Let $A_{\ell,\infty}(z)$ denote the finite-column, infinite-row matrix \begin{align*}
    A_{\ell,\infty}(z)\coloneqq(z^{y^j}-1)_{\substack{1\leq j\leq \ell\\ y\geq 2}}.
\end{align*} 
We first address the case that  this matrix  does not have full  rank $\ell$.

\begin{lemma}\label{lem:defRankImpliesUnity}
    Suppose $\theta\in \overline{\mathbb{Q}}^\times$ is a nonzero algebraic number and $A_{\ell,\infty}(\theta)$ has (column) rank $r$ with $r<\ell$. Then, there exists some positive integer $s$ such that $\theta^s=1$ (i.e. $\theta$ is a root of unity).
\end{lemma}

\begin{proof}
    Suppose that $A_{\ell,\infty}(\theta)$ has rank $r<\ell$ at some nonzero algebraic number $\theta$. Then, there exist algebraic numbers $c_1,\hdots,c_\ell$, not all zero, such that \begin{align*}
        \sum_{j=1}^\ell c_j(\theta^{y^j}-1)=0,\quad \forall y\geq 2.
    \end{align*} Let $C \coloneqq c_1+\hdots + c_\ell$, then $\sum_{j=1}^\ell c_j \theta^{y^j}=C$ for all $y\geq 2$. We will show that this forces $\theta$ to be a root of unity.

    Suppose by way of contradiction that $\theta$ is not a root of unity. Let $K$ be a number field containing $\theta$ and $c_1,\ldots,c_\ell$. By Lemma \ref{lem:kronecker}, there exists a place $v$ of $K$ such that $|\theta|_v\neq 1$; all absolute values below are taken with this place. We will use this to force a contradiction (and won't need any deeper structure about the place $v$).

    Consider first where $|\theta|_v>1$. Among the indices $j$ such that $c_j$ is nonzero, let $j_+$ be the largest. Dividing our constraint equation by $c_{j_+}\theta^{y^{j_+}}$, one obtains that \begin{align*}
        1+\sum_{1\leq j<j_+}\frac{c_j}{c_{j_+}}\theta^{y^j-y^{j_+}}=\frac{C}{c_{j_+}\theta^{y^{j_+}}}.
    \end{align*} By the triangle inequality for $|\cdot |_v$, then, \begin{align*}
        1\leq \Big|\frac{C}{c_{j_+}}\Big|_v|\theta|_v^{-y^{j_+}}+\sum_{1\leq j<j_+}\Big|\frac{c_j}{c_{j_+}}\Big|_v|\theta|_v^{y^j-y^{j_+}}.
    \end{align*} Since $|\theta|_v>1$ and maximality of $j_+$, by taking $y\rightarrow \infty$ we obtain a contradiction.

    Now consider where $|\theta|_v<1$. Considering the equation $\sum_{j=1}^\ell c_j\theta^{y^j}=C$ and taking $y\rightarrow \infty$, one observes that $C=0$. Thus, $\sum_{j=1}^\ell c_j\theta^{y^j}=0$. Now, among the indices $j$ such that $c_j$ is nonzero, let $j_-$ be the smallest. A similar argument as in the first case then gives a contradiction (dividing in this case by the lowest contributory monomial).
\end{proof}

We have shown that rank deficiency of the infinite matrix forces a nonzero algebraic number to be a root of unity. Now, in order to find a single finite cutoff \(M(\ell)\) for which the same conclusion holds, we express rank deficiency as the vanishing of polynomial minors and use the Noetherian property of $\mathbb{Q}[z]$ to reduce these infinitely many conditions to finitely many.  For a finite set $S\subset \mathbb{N}_{\geq 2}$ with $|S|=\ell$, let $M_S(z)\coloneqq(z^{y^j}-1)_{y\in S,\,1\leq j\leq \ell}$ denote the $\ell\times \ell$ submatrix of $A_{\ell,\infty}(z)$ with rows indexed by $S$, and let $\Delta(M_S(z))\in\mathbb{Z}[z]$ be its determinant. Throughout, $z$ is an indeterminate, and ``rank'' always refers to a specialization $z=\theta$ in a field. For a field $K$ and $\theta \in K$, a matrix with $\ell$ columns and entries in $K$ has rank less than $\ell$ if and only if all of its $\ell\times \ell$ minors vanish. Hence $A_{\ell,\infty}(\theta)$ has rank less than $\ell$ if and only if $\Delta(M_S(\theta))=0$ for every $S$.

Let $I\subseteq\mathbb{Q}[z]$ be the ideal generated by all the minors $\Delta(M_S(z))$.
Now, \( \mathbb{Q}[z] \) is a principal ideal domain. 
That is, every ideal is generated by a single element (yet the generator need not 
be one the minors).  
One may find a finite number of sets $S_1,...,S_t$ such that 
\begin{align*}
    I=(\Delta(M_{S_1}(z)),\hdots, \Delta(M_{S_t}(z))).
\end{align*}
Let $M(\ell)$ be the largest number occurring in any $S_1,\hdots, S_t$. 
Then all the generating minors are minors of the finite matrix $A_{\ell,M(\ell)}(z)$, whose rows are indexed by $2\leq y\leq M(\ell)$. Since each $S_i$ consists of $\ell$ distinct integers $\geq 2$, we have $M(\ell)\geq \ell+1$.

\begin{lemma}\label{lem:finitelyManyRowsSuffice}
    Let $K$ be a field of characteristic zero and let $\theta\in K$. Then the matrix  $A_{\ell,M(\ell)}(\theta)$ has rank less than $\ell$ if and only if $A_{\ell,\infty}(\theta)$ has rank less than $\ell$.
\end{lemma}

\begin{proof}
    ($\implies$) Suppose that $A_{\ell,M(\ell)}(\theta)$ has rank less than $\ell$. Then all of its $\ell\times \ell$ minors vanish, and in particular $\Delta(M_{S_i}(\theta))=0$ for $1\leq i\leq t$. Let $S$ be arbitrary. Since $\Delta(M_S(z))\in I$, we can write $\Delta(M_S(z))=\sum_{i=1}^t a_i\Delta(M_{S_i}(z))$ with $a_i\in\mathbb{Q}[z]$. Evaluation at $\theta$ is a ring homomorphism $\mathbb{Q}[z]\to K$ (because $K\supseteq\mathbb{Q}$), so $\Delta(M_S(\theta))=\sum_i a_i(\theta)\Delta(M_{S_i}(\theta))=0$. Thus every $\ell\times \ell$ minor of $A_{\ell,\infty}(\theta)$ vanishes, and $A_{\ell,\infty}(\theta)$ has rank less than $\ell$.

    ($\impliedby$) If $A_{\ell,\infty}(\theta)$ has rank less than $\ell$, then so does its submatrix $A_{\ell,M(\ell)}(\theta)$.
\end{proof}

The characteristic-zero hypothesis in the preceding lemma is essential: the coefficients $a_i$ may have denominators, so the argument does not pass to characteristic $q$. This is why the argument below works directly with the finitely many generators.

Lemmas \ref{lem:defRankImpliesUnity} and \ref{lem:finitelyManyRowsSuffice} together imply the following: 

\begin{lemma}\label{lem:fintiteDefRankImpliesUnity}
    Suppose $\theta\in \overline{\mathbb{Q}}$ is an algebraic number and $A_{\ell,M(\ell)}(\theta)$ has (column) rank $r$ with $r<\ell$. Then $\theta$ is zero or a root of unity.
\end{lemma}

In order to prove Proposition \ref{prop:uniformTorsionRank} it remains to bound, uniformly in \(q\), the possible prime orders of nonzero elements $\lambda$ at which the reduced matrix $A_{\ell, M(\ell)}(\lambda)$ has deficient rank over $\overline{\mathbb{F}}_q$. Such an element $\lambda$ must be a common root of the reduced generating minors $\Delta(M_{S_i}(z))$ mod $q$. We therefore study the common roots of these polynomials. Recall that the ideal $I \subseteq \mathbb{Q}[z]$, which is generated by all of the minors of the infinite matrix, is finitely generated by the minors $\Delta(M_{S_1}(z)),\hdots, \Delta(M_{S_t}(z)))$. Let $D(z)$ be their greatest common divisor over $\mathbb{Q}[z]$: \begin{align*}
    D(z)\coloneqq\text{gcd}(\Delta(M_{S_1}(z)),\hdots,\Delta(M_{S_t}(z))).
\end{align*} The roots of $D$ are the common roots of all the generating minors. By Lemma \ref{lem:defRankImpliesUnity}, these roots are either roots of unity (if nonzero), or zero. Since $D$ has finite degree, there are only finitely many roots. The irreducible factors of $D$ over $\mathbb{Q}$ are therefore $z$ and cyclotomic polynomials. Normalizing the greatest common divisor to be monic, we may write \begin{align*}
    D(z)=z^a\prod_{m\in \mathcal{E}_\ell}\Phi_m(z)^{e_m},
\end{align*} where $a\geq 0$ is an integer, $\Phi_m$ is the $m$-th cyclotomic polynomial, $e_m$ are positive integers, and $\mathcal{E}_\ell$ is some finite set of positive integers. Since each $\Phi_m$ is monic with integer coefficients, $D$ is a monic polynomial in $\mathbb{Z}[z]$.

We now use B\'ezout's identity to show that, outside finitely many characteristics $q$, a common root of the reduced generating minors must be a root of the reduction of $D(z)$ mod $q$. Write $\Delta(M_{S_i}(z))=D(z)H_i(z)$ for $1\leq i\leq t$. Since $D$ is monic in $\mathbb{Z}[z]$, division in $\mathbb{Z}[z]$ gives $H_i\in\mathbb{Z}[z]$. By the definition of $D$, the polynomials $H_1,\ldots,H_t$ have greatest common divisor $1$ in $\mathbb{Q}[z]$. Thus by B\'ezout's identity (which applies since $\mathbb{Q}[z]$ is a PID), there exist $U_1,\ldots,U_t\in\mathbb{Q}[z]$ such that $\sum_i H_iU_i=1$. Clearing denominators gives $V_1,\ldots,V_t\in\mathbb{Z}[z]$ and a positive integer $N$, depending only on $\ell$, such that
\begin{align}\label{eq:charZeroFundamental}
    \sum_{i=1}^t H_i(z)V_i(z)=N.
\end{align}
Multiplying by $D(z)$ gives
\begin{align*}
    \sum_{i=1}^t\Delta(M_{S_i}(z))V_i(z)=N D(z).
\end{align*}

Let $q\nmid N$ be prime, and let $\lambda\in\overline{\mathbb{F}}_q$ be a common root of the reduced generating minors. Reducing this identity modulo $q$ and evaluating at $\lambda$ gives $0=\overline{N}\,\overline{D}(\lambda)$, where a bar denotes reduction modulo $q$. Since $\overline{N}\neq0$, we obtain $\overline{D}(\lambda)=0$. We call the primes $q\nmid N$ the \emph{good characteristics}; the finitely many primes dividing $N$ will be treated separately. For a nonzero $\lambda$ in a good characteristic, the cyclotomic factorization of $D$ now bounds its order, as the following lemma shows.

\begin{lemma}\label{lem:goodCharacteristics}
    Suppose $q\nmid N$, and that $\lambda \in \overline{\mathbb{F}}_q^\times$ is a root of the reduction of $D(z)=z^a \prod_{m\in \mathcal{E}_\ell}\Phi_m(z)^{e_m}$ modulo $q$ (call this $\overline{D}(z)$). Then, there exists some $m\in \mathcal{E}_\ell$ such that $\lambda^m=1$.
\end{lemma}

\begin{proof}
    Since $\lambda\neq 0$, one has that $\lambda$ is the root of some $\overline{\Phi_m}$, with $m\in \mathcal{E}_\ell$. Since $\Phi_m(z)|(z^m-1)$, the same divisibility holds after reducing modulo $q$, and so $\lambda^m=1$.
\end{proof}

We now turn our view to the \textit{exceptional characteristics}, namely those primes $q$ such that $q|N$. The exceptional characteristics require one more property of the minors: they never vanish identically modulo a prime.

\begin{lemma}\label{lem:minorsMonic}
    Let $2\leq y_1<y_2<\cdots<y_\ell$ be integers. Then the minor
    \begin{align*}
        \Delta(z) \coloneqq \det\big(z^{y_r^j}-1\big)_{1\leq r,j\leq \ell}\in \mathbb{Z}[z]
    \end{align*}
    is monic of degree $\sum_{r=1}^\ell y_r^{r}$. In particular, for every prime $q$ its reduction $\overline{\Delta}\in \mathbb{F}_q[z]$ is a nonzero polynomial of the same degree.
\end{lemma}

\begin{proof}
    Expanding the determinant,
    \begin{align*}
        \Delta(z) = \sum_{\sigma\in \textup{Sym}_\ell}\sgn(\sigma)\prod_{r=1}^\ell\big(z^{y_r^{\sigma(r)}}-1\big).
    \end{align*}
    The term indexed by $\sigma$ is a polynomial of degree $e(\sigma)\coloneqq\sum_{r=1}^\ell y_r^{\sigma(r)}$ with leading coefficient $\sgn(\sigma)$. We claim that $e(\sigma)<e(\mathrm{id})$ for every $\sigma\neq \mathrm{id}$. Then the monomial $z^{e(\mathrm{id})}$ occurs only in the term $\sigma = \mathrm{id}$, with coefficient $+1$, which proves the lemma.

    We prove the claim by induction on the number of inversions of $\sigma$. If $\sigma\neq\mathrm{id}$, there are $r<s$ with $a\coloneqq\sigma(s)<b\coloneqq\sigma(r)$. Let $\tau$ agree with $\sigma$ except that $\tau(r)=a$ and $\tau(s)=b$. Then $\tau$ has fewer inversions than $\sigma$, and 
    \begin{align*}
        e(\tau)-e(\sigma) = \big(y_s^{b}-y_s^{a}\big)-\big(y_r^{b}-y_r^{a}\big) = \phi(y_s)-\phi(y_r), \qquad \phi(t)\coloneqq t^{a}\big(t^{b-a}-1\big).
    \end{align*}
    The function $\phi$ is strictly increasing on $[1,\infty)$, and $y_s>y_r\geq 2$, so $e(\tau)>e(\sigma)$. By induction $e(\tau) \leq e(\mathrm{id})$, and hence $e(\sigma)<e(\mathrm{id})$.
\end{proof}

\begin{lemma}\label{lem:exceptionalCharacteristics}
    There is a positive integer $c_\ell$, depending only on $\ell$, with the following property. Suppose $q\mid N$ is prime and that $\lambda\in \overline{\mathbb{F}}_q^\times$ is a common root of $\overline{\Delta(M_{S_1}(z))},\hdots,\overline{\Delta(M_{S_t}(z))}$. Then $\lambda^{c_\ell}=1$.
\end{lemma}

\begin{proof}
    Fix a prime $q\mid N$, and let $G_q(z)$ be the greatest common divisor in $\mathbb{F}_q[z]$ of the reductions $\overline{\Delta(M_{S_1}(z))},\hdots,\overline{\Delta(M_{S_t}(z))}$. By Lemma \ref{lem:minorsMonic}, each of these reductions is a nonzero polynomial. Hence $G_q\neq 0$, and 
    \begin{align*}
        d_q\coloneqq\deg G_q\leq \deg \overline{\Delta(M_{S_1}(z))} = \deg \Delta(M_{S_1}(z)).
    \end{align*}
    (Without this, $G_q$ could be the zero polynomial, every $\lambda$ would be a common root, and the argument would give no information.) Every common root $\lambda$ is a root of $G_q$, and the roots of the nonzero polynomial $G_q$ lie in extensions of $\F_q$ of degree at most $d_q$: that is, $\lambda\in \mathbb{F}_{q^r}$ for some $1\leq r\leq d_q$. If $\lambda\neq 0$, then $\lambda^{q^r-1}=1$. Set
    \begin{align*}
        c_\ell\coloneqq \operatorname{lcm}\big\{q^r-1:\ q \text{ prime},\ q\mid N,\ 1\leq r\leq d_q\big\}.
    \end{align*}
    This is a finite least common multiple, since $N$ has finitely many prime divisors, and it depends only on $\ell$, since $N$, the sets $S_1,\ldots,S_t$ and hence each $d_q$ depend only on $\ell$. Every $\lambda$ as in the statement satisfies $\lambda^{c_\ell}=1$.
\end{proof}

We now conclude our qualitative proof of Proposition \ref{prop:uniformTorsionRank}.

\begin{proof}[Qualitative proof of Proposition \ref{prop:uniformTorsionRank}] Choose $M=M(\ell)$ as defined before Lemma \ref{lem:finitelyManyRowsSuffice}, and choose 
    \begin{align*}
        P > \max\Big(\mathcal{E}_\ell \cup \{c_\ell\}\Big),
    \end{align*}
    with $c_\ell$ as in Lemma \ref{lem:exceptionalCharacteristics}. Both quantities depend only on $\ell$. Let $p>P$ and $q\neq p$ be primes, and let $\lambda\in\overline{\mathbb{F}}_q^\times$ have multiplicative order $p$. Suppose, for contradiction, that $A_{\ell,M}(\lambda)$ has rank less than $\ell$. Then every $\ell\times \ell$ minor of $A_{\ell,M}(\lambda)$ vanishes. The generating sets $S_1,\ldots,S_t$ lie in $\{2,\ldots,M(\ell)\}$, and the entries $z^{y^j}-1$ have integer coefficients, so evaluating the reduction $\overline{\Delta(M_{S_i}(z))}$ at $\lambda$ gives a $\ell \times \ell$ minor of $A_{\ell,M}(\lambda)$. Hence $\lambda$ is a common root of $\overline{\Delta(M_{S_1}(z))},\ldots,\overline{\Delta(M_{S_t}(z))}$.

    If $q\nmid N$, the discussion above gives $\overline{D}(\lambda)=0$, and Lemma \ref{lem:goodCharacteristics} gives $\lambda^m=1$ for some $m\in\mathcal{E}_\ell$. Since $\lambda$ has order $p$, this forces $p\mid m$, so $p\leq m<P$, a contradiction. If $q\mid N$, Lemma \ref{lem:exceptionalCharacteristics} gives $\lambda^{c_\ell}=1$, so $p\mid c_\ell$ and $p\leq c_\ell<P$, again a contradiction. Hence $A_{\ell,M}(\lambda)$ has rank $\ell$.
\end{proof}

\subsection{Quantitative proof of Proposition \ref{prop:uniformTorsionRank}}\label{sec:quantitativeLem}

We now give explicit bounds for the parameters $M(\ell)$ and $P(\ell)$ in Proposition \ref{prop:uniformTorsionRank}. We begin by proving Lemma \ref{thm:quantitativeRankDeficiency} which is a quantitative version of Lemma \ref{lem:fintiteDefRankImpliesUnity} from Subsection \ref{sec:qualitativeLem}, and then we bound the quantities arising in the passage to finite characteristic. Our first step is the following determinant estimate.

\begin{lemma}\label{lem:smallMinors}
    Let $\ell\geq1$ and $L\geq\ell+1$ be integers, and let $\alpha\in\C$. Suppose that
    \begin{equation*}
        \det \left(\alpha^{(L - \ell + r)^j}-1\right)_{1 \leq j \leq \ell,\  1 \leq r \leq \ell} = \det\begin{bmatrix}
            \alpha^{(L - \ell + 1)}-1 & \alpha^{(L - \ell + 1)^2}-1 & \cdots & \alpha^{(L - \ell + 1)^\ell}-1\\
            \alpha^{(L - \ell + 2)}-1 & \alpha^{(L - \ell + 2)^2}-1 & \cdots & \alpha^{(L - \ell + 2)^\ell}-1\\
            \vdots & \vdots & \ddots & \vdots\\
            \alpha^{L}-1 & \alpha^{L^2}-1 & \cdots & \alpha^{L^\ell}-1
        \end{bmatrix}=0.
    \end{equation*}
    Then
    \begin{align*}
        |\alpha|\leq\exp\!\left(\frac{\ell\log(\ell+1)}{L-\ell}\right).
    \end{align*}
\end{lemma}

\begin{proof}
    The conclusion is immediate if $|\alpha|\leq1$, so suppose that $|\alpha|>1$. Let $\Delta(z)$ be the polynomial defined by
      \[ \Delta(z) = \det \left(z^{(L - \ell + r)^j}-1\right)_{1 \leq j \leq \ell,\  1 \leq r \leq \ell}, \]
      so by assumption we have $\Delta(\alpha) = 0$. 
      We will show that $\Delta(z)$ has leading term $z^d$, and every other term in the polynomial has degree at most $d - L + \ell$, where $d := \sum_{r=1}^\ell (L - \ell + r)^r $. Expanding the determinant gives
    \begin{align} \label{deltaDef}
        \Delta(z)=\sum_{\sigma\in S_\ell}\sgn(\sigma)
        \prod_{r=1}^\ell \big(z^{(L - \ell + r)^{\sigma(r)}}-1\big),
    \end{align}
    where $S_\ell$ is the set of permutations of $\{1,\ldots,\ell\}$. Note that for $\sigma \in S_\ell$, the product 
      \begin{equation}\label{permProd} \prod_{r=1}^\ell\big(z^{(L - \ell + r)^{\sigma(r)}}-1\big) \end{equation}
  has degree $e(\sigma)\coloneqq\sum_{r = 1}^\ell (L - \ell + r)^{\sigma(r)}$. We claim that $e(\sigma)\leq \left(\sum_{r=1}^\ell (L - \ell + r)^r \right)-L + \ell $ whenever $\sigma\neq\mathrm{id}$. To see this, choose $1\leq r<\ell$ such that $\sigma(r+1)<\sigma(r)$. Let $\tau\in S_\ell$ be the transposition which swaps $r$ and $r+1$, so that $\sigma\circ\tau$ is obtained from $\sigma$ by swapping these two values. Then, letting $v = L - \ell + r$, we have
    \begin{align*}
        e(\sigma\circ\tau)-e(\sigma)
        &=\big((v+1)^{\sigma(r)}-(v+1)^{\sigma(r+1)}\big)
          -\big(v^{\sigma(r)}-v^{\sigma(r+1)}\big)\\
        &=(v+1)^{\sigma(r+1)}
          \big((v+1)^{\sigma(r)-\sigma(r+1)}-v^{\sigma(r)-\sigma(r+1)}\big)\\
        &\quad+\big((v+1)^{\sigma(r+1)}-v^{\sigma(r+1)}\big)
          \big(v^{\sigma(r)-\sigma(r+1)}-1\big)\\
        &\geq(v+1)^{\sigma(r+1)}\geq v\geq L-\ell.
    \end{align*}
    In the above formulae we use the fact that $v\geq L-\ell\geq1$ and $\sigma(r)>\sigma(r+1)\geq1$. Repeatedly applying transpositions in this way to $\sigma$ will eventually result in the identity permutation, so 
      \[ e(\sigma) \leq e(\text{id}) - (L-\ell) = \left(\sum_{r=1}^\ell (L - \ell + r)^r \right)-L + \ell . \]

 The product 
      \begin{equation} \label{idProd} \prod_{r=1}^\ell \big(z^{(L - \ell + r)^{r}}-1\big)\end{equation}
      has leading term $z^d$. Every other monomial in this product is obtained by choosing $-1$ from at least one factor, lowering the degree by at least $\min_{1 \leq r \leq \ell} (L - \ell + r)^r\geq L - \ell $, thus each term other than the leader term has degree $\leq d - L + \ell$. 

It follows from our observations that $\Delta(z)=z^d + R(z)$ where $\deg(R(z)) \leq d - L + \ell$. Let $r_0,r_1,\ldots,r_{d-L + \ell }$ be the coefficients of $R(z)$, so $R(z) = \sum_{j=0}^{d - L + \ell} r_j z^j$. Expanding all $\ell!$ products in formula \eqref{deltaDef} gives $2^\ell\ell!$ monomial terms with coefficients $\pm 1$ before collecting like powers, so $\sum_{j=0}^{d - L + \ell } |r_j|\leq2^\ell\ell!$.

    Since $\Delta(\alpha)=\alpha^d + R(\alpha) = 0$, it follows that
    \begin{align*}
       |\alpha|^d
        &= |R(\alpha)| = \left| \sum_{j=0}^{d - L + \ell} r_j \alpha^j\right| \leq|\alpha|^{d - L + \ell} \sum_{j=0}^{d - L + \ell}|r_j|  \leq2^\ell\ell!|\alpha|^{d-L + \ell}.
    \end{align*}
    By the AM-GM inequality,
    \begin{align*}
        \ell!=\prod_{r=1}^\ell r\leq\left(\frac{\ell+1}{2}\right)^\ell,
    \end{align*}
    hence $2^\ell \ell! \leq (\ell + 1)^\ell$, and so $|\alpha|^d \leq (\ell + 1)^\ell |\alpha|^{d - L + \ell}$. Dividing by $|\alpha|^{d-L+\ell}$ and taking logarithms therefore gives
    \begin{align*}
        \log|\alpha|\leq\frac{\log(2^\ell\ell!)}{L-\ell}
        \leq\frac{\ell\log(\ell+1)}{L-\ell},
    \end{align*}
    as desired.
\end{proof}

To turn the estimate of Lemma \ref{lem:smallMinors} into a bound on the constant $M(\ell)$ from Lemma \ref{lem:fintiteDefRankImpliesUnity}, we use the following theorem of Dimitrov, which gives a lower bound for the largest absolute value of a conjugate of a nonzero algebraic integer that is not a root of unity and has degree greater than $1$.

\begin{theorem}[{\cite[Theorem 1]{VD}}]\label{thm:dimitrov}
    Suppose $Q(z)\in\Z[z]$ is a monic irreducible polynomial of degree $d>1$. If $Q$ is not cyclotomic, then
    \begin{align*}
        \max_{Q(\beta)=0}|\beta|\geq2^{1/(4d)}.
    \end{align*}
\end{theorem}

We now use Lemma \ref{lem:smallMinors} and Theorem \ref{thm:dimitrov} to obtain the following quantitative version of Lemma \ref{lem:fintiteDefRankImpliesUnity}.

\begin{lemma}\label{thm:quantitativeRankDeficiency}
    There exists an absolute constant $\kappa>0$ such that the following holds. Let $\ell\geq1$ be an integer, let $M = M(\ell) \geq2$ be an integer satisfying $M(\ell)\geq e^{\kappa\ell\log\ell}$, and let $\alpha\in\overline{\mathbb{Q}}$. If
    \begin{equation}\label{eq:quantitativeRankDeficiency}
        \operatorname{rank}A_{\ell,M}(\alpha)
        =\operatorname{rank}\begin{bmatrix}
            \alpha^2-1 & \alpha^{2^2}-1 & \cdots & \alpha^{2^\ell}-1\\
            \alpha^3-1 & \alpha^{3^2}-1 & \cdots & \alpha^{3^\ell}-1\\
            \vdots & \vdots & \ddots & \vdots\\
            \alpha^M-1 & \alpha^{M^2}-1 & \cdots & \alpha^{M^\ell}-1
        \end{bmatrix}<\ell,
    \end{equation}
    then $\alpha$ is $0$ or a root of unity.
\end{lemma}

\begin{proof}
    If $\ell=1$, then $M\geq2$, and rank deficiency forces $\alpha^2-1=0$. Thus the conclusion holds in this case. Assume henceforth that $\ell\geq2$, and set
    \begin{equation}\label{eq:determinantCutoff}
        \begin{aligned}
              d_\ell\coloneqq\sum_{r=1}^\ell(r+1)^r,\quad N(\ell)\coloneqq\ell+\frac{4d_\ell\,\ell\log(\ell+1)}{\log2}.
        \end{aligned}
    \end{equation}
    Since
    \begin{align*}
        (\ell+1)^\ell\leq d_\ell
        \leq\sum_{r=1}^\ell(\ell+1)^r
        <\left(1+\frac1\ell\right)(\ell+1)^\ell,
    \end{align*}
    we have $
        \log N(\ell)
        =\ell\log(\ell+1)+O(\log(\ell+1))=O(\ell\log\ell)$. 
    Consequently, $N(\ell)<e^{\kappa\ell\log\ell}$ for some absolute constant $\kappa>0$, and so it suffices to prove the assertion of the lemma for every integer $M>N(\ell)$. Let $\Delta_1(z),\ldots,\Delta_t(z) \in \Z[z]$ be all the $\ell\times\ell$ minors of $A_{\ell,M}(z)$, and let $D(z)$ be their monic greatest common divisor over $\mathbb{Q}[z]$. By Lemma \ref{lem:minorsMonic}, the minors are monic integer polynomials. Write $\Delta_1=DH$ with $H\in\mathbb{Q}[z]$ monic. Clearing denominators and removing the gcd of the coefficients gives $D=A/a$ and $H=B/b$, where $A,B\in\mathbb{Z}[z]$ are primitive (their coefficients have gcd $1$), with positive leading coefficients $a,b$. Then $AB=ab\Delta_1$. By Gauss's lemma, $AB$ is primitive, whereas the coefficients of $ab\Delta_1$ have gcd $ab$, since $\Delta_1$ is monic. Thus $ab=1$, so $a=b=1$, therefore $D(z) \in\mathbb{Z}[z]$.

    If $A_{\ell,M}(\alpha)$ has rank less than $\ell$, all the minors vanish at $\alpha$, and hence $D(\alpha)=0$ since the roots of $D(z)$ are precisely the common roots of $\Delta_1(z),\ldots,\Delta_t(z)$. It remains to show that every nonzero root of $D$ is a root of unity.

    Let $\beta\neq0$ be a root of $D$. Since $D\in\mathbb{Z}[z]$ is monic, $\beta$ is an algebraic integer. Let $m\in\mathbb{Q}[z]$ be its monic minimal polynomial over $\mathbb{Q}$, and write $d=\deg m$. Since $D(\beta)=0$, the minimal polynomial $m$ divides $D$ in $\mathbb{Q}[z]$; the same argument we used earlier applied to this monic factorization gives $m\in\mathbb{Z}[z]$. Every conjugate $\gamma$ of $\beta$ is a root of $m$, hence of $D$, and so is a common root of the minors. Lemma \ref{lem:smallMinors}, applied with $L=M$, therefore gives
    \begin{equation}\label{eq:conjugateUpperBound}
        \log|\gamma|\leq\frac{\ell\log(\ell+1)}{M-\ell}<\frac{\log2}{4d_\ell}.
    \end{equation}
    On the other hand, $D$ divides the minor with rows $2,3,\ldots,\ell+1$. By Lemma \ref{lem:minorsMonic},
    \begin{align*}
        d\leq\deg D\leq\sum_{r=1}^\ell(r+1)^r=d_\ell.
    \end{align*}
    If $d>1$ and $m$ were not cyclotomic, Theorem \ref{thm:dimitrov} would give a conjugate $\gamma$ with
    \begin{align*}
        \log|\gamma|\geq\frac{\log2}{4d}\geq\frac{\log2}{4d_\ell},
    \end{align*}
    contradicting \eqref{eq:conjugateUpperBound}. If $d=1$, then $\beta$ is a nonzero integer, and \eqref{eq:conjugateUpperBound} gives $|\beta|<2$, so $\beta=\pm1$. Thus every nonzero root of $D$ is a root of unity, as required for the cutoff \eqref{eq:determinantCutoff}.
\end{proof}

Lemma \ref{thm:quantitativeRankDeficiency} supplies an explicit choice of $M(\ell)$. To obtain a bound for $P(\ell)$, we make the passage to finite characteristic quantitative. To be precise, we wish to find an explicit $P(\ell)$ such that for every pair of distinct primes $p,q$ with $p>P(\ell)$ and every $\lambda\in\overline{\mathbb{F}}_q^\times$ of order $p$, the matrix $A_{\ell,M(\ell)}(\lambda)$ has rank $\ell$. Equivalently, we must show that any such $\lambda$ causing rank deficiency has prime order at most $P(\ell)$, uniformly in $q$. Such an element $\lambda$ must be a common root of the minors of $A_{\ell, M(\ell)}$ reduced mod $q$. We therefore study the common roots of these polynomials, similar to how we proceeded in the qualitative argument of Subsection \ref{sec:qualitativeLem}. 

Fix $M = M(\ell)$, as given by Lemma \ref{thm:quantitativeRankDeficiency}. Let $\Delta_1(z),\ldots,\Delta_t(z)$ be all the $\ell\times\ell$ minors of $A_{\ell,M}(z)$, with $\Delta_1$ corresponding to the first $\ell$ rows and with the rows of every minor in increasing order. These are monic integer polynomials with $\deg\Delta_1(z)=d_\ell$ and $ \deg\Delta_i(z) \leq n_\ell$ for $2 \leq i \leq t$, where
  \[ d_\ell \coloneqq \sum_{r=1}^\ell (r+1)^r ,\quad n_\ell \coloneqq \sum_{r=1}^\ell (M - \ell + r)^r. \]
Also, similar to how we argued in the proof of Lemma \ref{lem:smallMinors}, expanding a determinant $\Delta_i(z)$ using the Leibniz formula gives $2^\ell\ell!$ monomial terms with coefficient $\pm 1$ before collecting like powers, so the sum of the absolute values of the coefficients of any polynomial $\Delta_i(z)$ is at most $2^\ell\ell!$.

Let $D(z)$ be the monic greatest common divisor of $\Delta_1(z),\ldots, \Delta_t(z)$ in $\mathbb{Q}[z]$. Note that, since $D(z) \mid \Delta_1(z)$, we have $\deg(D(z)) \leq \deg(\Delta_1(z)) = d_\ell$. As in the proof of Lemma \ref{thm:quantitativeRankDeficiency}, $D(z)\in\mathbb{Z}[z]$. That lemma and the choice of $M$ show that every nonzero root of $D$ is a root of unity. Thus
    \begin{align*}
    D(z)=z^a\prod_{m\in\mathcal{E}}\Phi_m(z)^{e_m},
    \end{align*}
for some finite set $\mathcal{E} \subseteq \mathbb{N}$ and positive integers $e_m$. 

Similar to how we proceeded in Subsection \ref{sec:qualitativeLem} we now use a quantitative version of B\'ezout's identity to show that, outside finitely many characteristics $q$, a common root of the reduced generating minors must be a root of the reduction of $D(z)$ mod $q$. In order to do this, 
we will find polynomials $V_1(z),\ldots,V_t(z) \in \Z[z]$ and positive integer $1\leq B\leq \big(2^\ell\ell!\big)^{\,n_\ell+d_\ell}$ such that 
    \begin{equation*}
        \sum_{i=1}^t\Delta_i(z)V_i(z)=B D(z). 
    \end{equation*}
      First, since $D(z) = \gcd(\Delta_1(z),\ldots,\Delta_t(z))$ in $\mathbb{Q}[z]$, by B\'ezout's lemma there are polynomials $U_1(z),\ldots,U_t(z)\in \mathbb{Q}[z]$ such that $D(z)=\sum_i U_i(z)\Delta_i(z)$. We can choose these polynomials so that $\deg (U_i(z))<\deg(\Delta_1(z)) = d_\ell$ for every $i\geq2$. To see this, for each $2 \leq i \leq t$, polynomial division gives $Q_i(z), R_i(z) \in \mathbb{Q}[z]$ such that $ U_i=Q_i\Delta_1+R_i$ and $\deg R_i<\deg\Delta_1=d_\ell$. 
    Substituting into the B\'ezout identity and collecting the multiples of $\Delta_1$, we obtain
    \begin{align*}
        D=\left(U_1+\sum_{i=2}^t Q_i\Delta_i\right)\Delta_1
          +\sum_{i=2}^t R_i\Delta_i.
    \end{align*}
    Thus, replacing $U_1$ by $U_1+\sum_{i=2}^t Q_i\Delta_i$ and $U_i$ by $R_i$ for $i\geq2$ preserves the identity and gives the desired degree bounds $\deg (U_i(z))<d_\ell$ for $2 \leq i \leq t$. With these choices, we have
    \begin{align} \label{bezEq}
        U_1\Delta_1=D-\sum_{i=2}^t U_i\Delta_i,
    \end{align}
Since $\deg(D(z)) \leq \deg(\Delta_1(z)) = d_\ell$, and for all $2\leq i \leq t$ we have $\deg(U_i(z)) < d_\ell$, and $\deg(\Delta_i(z)) \leq n_\ell$, the right-hand side of \eqref{bezEq} has degree less than $n_\ell+d_\ell$. Since $\deg\Delta_1=d_\ell$, it follows that $\deg(U_1(z))<n_\ell$. 

Let $u_{1,0}, u_{1,1},\ldots,u_{1,n_\ell-1} \in \mathbb{Q}$ be the coefficients of $U_1(z)$, i.e.,
  \[ U_1(z) = \sum_{j=0}^{n_\ell-1} u_{1,j}z^j, \]
  and for $2 \leq i \leq t$ let $u_{i,0},u_{i,1},\ldots,u_{i,d_\ell-1} \in \mathbb{Q}$ be the coefficients for $U_i(z)$, i.e.,
    \[ U_i(z) = \sum_{j=0}^{d_\ell-1} u_{i,j}z^j. \]
Since $D(z) = U_1(z) \Delta_1(z) + \cdots + U_t(z) \Delta_t(z)$, it follows that
  \begin{equation} \label{uCoeffProp} D(z) = \sum_{j=0}^{n_\ell -1} u_{1,j} \cdot (z^j \Delta_1(z)) + \sum_{i=2}^t\sum_{j=0}^{d_\ell -1} u_{i,j} \cdot (z^j \Delta_i(z)) . \end{equation}

Let $C = (C_{r,{(i,j)}})$ be a matrix with rows indexed by $0 \leq r < n_\ell + d_\ell$ and columns indexed by pairs $(i,j)$ with $i = 1$ and $0 \leq j < n_\ell$, or with $2 \leq i \leq t$ and $0 \leq j < d_\ell$. Define $C$ so that the entries in its $(i,j)$-column are the coefficients of the polynomial $z^j \Delta_i(z)$. To be precise, let the entries of $C$ be defined to satisfy the following: for $0 \leq j < n_\ell$, 
  \[  z^j \Delta_1(z) = \sum_{r=0}^{n_\ell + d_\ell-1} C_{r, (1,j)} z^r, \]
  and for $2 \leq i \leq t$, for $0 \leq j < d_\ell$,
    \[z^j \Delta_i(z) = \sum_{r=0}^{n_\ell + d_\ell-1} C_{r, (i,j)} z^r. \]

     Every column has Euclidean norm at most $2^\ell\ell!$, since shifting a polynomial does not change the sum of the absolute values of its coefficients.

Let us write $D(z) = \sum_{j=0}^{d_\ell} b_j z^j$ where $b_0,\ldots,b_{d_\ell} \in \mathbb{Z}$. By formula \eqref{uCoeffProp} there exists some vector $x$ whose entries are all of the form $u_{i,j}$ and another vector $\mathbf{b}$ such that $Cx = \mathbf{b}$, where $\mathbf{b} = (b_0,b_1,\ldots,b_{d_\ell}, 0, \ldots, 0)$, where $b_0,b_1,\ldots,b_{d_\ell} \in \mathbb{Z}$ are the coefficients of $D(z)$. 

Let $\rho$ be the rank of $C$. Note that, since $C$ has $n_\ell + d_\ell$ rows, we must have $\rho \leq n_\ell + d_\ell$. Define a new matrix $A$ by choosing $\rho$ linearly independent columns from $C$. It follows that there exists a unique vector $v$ such that $Av = \mathbf{b}$. Define $A_0$ by choosing $\rho $ rows of $A$ so that $A_0$ is invertible. Let $\mathbf{b}_0$ be the vector consisting of the corresponding entries of $\mathbf{b}$, so that $\mathbf{b}_0=A_0v$ and $\mathbf{b}_0$ has integer coordinates. Since $A_0$ has rational entries, its inverse $A_0^{-1}$ must also have rational entries, so, since $v = A_0^{-1} \mathbf{b}_0$ and $\mathbf{b}_0$ has integer coordinates, it follows that the coordinates of $v$ are all rational numbers. By Cramer's rule, the $i$th coordinate of $v$ is given by
  \[ v_i= \frac{\det(A_0^{(i)})}{\det(A_0)}, \]
where the matrix $A_0^{(i)}$ is obtained by replacing the $i$th column of $A_0$ with the vector $\mathbf{b}_0$. Let $B = |\det(A_0)|$. By Hadamard's inequality, $|\det(A_0)|$ is bounded above by the product of lengths of the columns of $A_0$. Since $A_0$ was constructed from the larger matrix $C$, it follows that each column of $A_0$ has length at most $2^\ell \ell !$, thus
  \[ B = |\det(A_0)| \leq (2^\ell \ell!)^\rho \leq (2^\ell \ell!)^{n_\ell + d_\ell}. \]

Then the coordinates of $Bv$ are all integers, and we have $A(Bv) = B\mathbf{b}$. For each column $(i,j)$ of $C$ retained in $A$, let $v_{i,j}$ be the corresponding coordinate of $Bv$, and set $v_{i,j}=0$ for every column not retained. These are integers. Since each column of $C$ consists of the coefficients of $z^j\Delta_i(z)$, we obtain
    \[ B  D(z) = \sum_{j=0}^{n_\ell -1} v_{1,j} \cdot (z^j \Delta_1(z)) + \sum_{i=2}^t\sum_{j=0}^{d_\ell -1} v_{i,j} \cdot (z^j \Delta_i(z)) .\]
  Thus, letting $V_i(z) = \sum_j v_{i,j}z^j$ for $1 \leq i \leq t$, it follows that $V_1(z),\ldots,V_t(z) \in \Z[z]$ and
    \begin{equation} \label{eq:quantitativeBezout} \sum_{i=1}^t\Delta_i(z)V_i(z)=B D(z), \end{equation}
as desired. 

The identity \eqref{eq:quantitativeBezout} now lets us treat the common roots of the reduced minors in two cases, just as in the qualitative argument. We call the primes $q\nmid B$ the \emph{good characteristics} and the primes $q\mid B$ the \emph{exceptional characteristics}. Throughout the following two lemmas, a bar denotes reduction modulo the prime $q$ under consideration.

In a good characteristic, evaluating \eqref{eq:quantitativeBezout} at a common root of the reduced minors forces it to be a root of $\overline D$. The cyclotomic factorization of $D$ then gives a bound on its prime order in terms of $d_\ell$.

\begin{lemma}\label{lem:quantitativeGoodCharacteristics}
    Suppose $q\nmid B$ is prime and $\lambda\in\overline{\mathbb{F}}_q^\times$ is a common root of $\overline{\Delta_1},\ldots,\overline{\Delta_t}$. If $\lambda$ has prime multiplicative order $p$, then $p\leq d_\ell+1$.
\end{lemma}

\begin{proof}
    Reducing \eqref{eq:quantitativeBezout} modulo $q$ and evaluating at $\lambda$ gives $0=\overline B\,\overline D(\lambda)$. Since $q\nmid B$, we have $\overline D(\lambda)=0$. Since $\lambda\neq0$, the factorization of $D$ shows that $\lambda$ is a root of $\overline{\Phi_m}$ for some $m\in\mathcal{E}$. The divisibility $\Phi_m(z)\mid z^m-1$ persists modulo $q$, so $\lambda^m=1$. Thus $p\mid m$.

    Since $\Phi_m$ divides $D$, we have $\varphi(m)=\deg\Phi_m\leq\deg D\leq d_\ell$. Write $m=p^k u$ with $k\geq1$ and $p\nmid u$. Then
    \begin{align*}
        p-1\leq p^{k-1}(p-1)\varphi(u)=\varphi(m)\leq d_\ell,
    \end{align*}
    which gives $p\leq d_\ell+1$.
\end{proof}

In an exceptional characteristic, the same identity gives no information about $\overline D(\lambda)$. Instead, we use the bound on $B$ to bound $q$, and the degree of the monic minor $\Delta_1$ to bound the extension of $\mathbb{F}_q$ containing $\lambda$. Together these bound its multiplicative order.

\begin{lemma}\label{lem:quantitativeExceptionalCharacteristics}
    Suppose $q\mid B$ is prime and $\lambda\in\overline{\mathbb{F}}_q^\times$ is a common root of $\overline{\Delta_1},\ldots,\overline{\Delta_t}$. Then the multiplicative order of $\lambda$ is at most
    \begin{align*}
        \big(2^\ell\ell!\big)^{d_\ell(n_\ell+d_\ell)}-1.
    \end{align*}
\end{lemma}

\begin{proof}
    Since $q\mid B$, we have $q\leq B\leq(2^\ell\ell!)^{n_\ell+d_\ell}$. By Lemma \ref{lem:minorsMonic}, the reduction $\overline{\Delta_1}$ is monic of degree $d_\ell$. The minimal polynomial of $\lambda$ over $\mathbb{F}_q$ divides $\overline{\Delta_1}$, so its degree $r$ satisfies $1\leq r\leq d_\ell$. Hence $\lambda\in\mathbb{F}_{q^r}^\times$, and its multiplicative order divides $q^r-1$. In particular, that order is at most
    \begin{align*}
        q^r-1\leq B^{d_\ell}-1
        \leq\big(2^\ell\ell!\big)^{d_\ell(n_\ell+d_\ell)}-1,
    \end{align*}
    as required.
\end{proof}

We now conclude the quantitative proof of Proposition \ref{prop:uniformTorsionRank} by choosing $P(\ell)$ to cover both cases and checking the asserted growth rates.

\begin{proof}[Quantitative proof of Proposition \ref{prop:uniformTorsionRank}]
    For $\ell\geq2$, take $M=M(\ell)$ to be an integer satisfying $e^{\kappa\ell\log\ell}\leq M\leq2e^{\kappa\ell\log\ell}$, with $\kappa$ as provided by Lemma \ref{thm:quantitativeRankDeficiency}. With $d_\ell$ and $n_\ell$ as defined above, set
    \begin{equation}\label{eq:quantitativeRankParameters}
        P(\ell)\coloneqq\big(2^\ell\ell!\big)^{d_\ell(n_\ell+d_\ell)}.
    \end{equation}
    For $\ell=1$, set $M(1)=P(1)=2$.  These choices have the growth rates asserted in the proposition. Indeed, $M(\ell)=\exp(O(\ell\log\ell))$, and $d_\ell\leq\ell(\ell+1)^\ell=\exp(O(\ell\log\ell))$,  $n_\ell\leq\ell M(\ell)^\ell=\exp(O(\ell^2\log\ell)).$
    Since $\log(2^\ell\ell!)\leq\ell\log(\ell+1)$, it follows that $ \log P(\ell) =d_\ell(n_\ell+d_\ell)\log(2^\ell\ell!) =\exp(O(\ell^2\log\ell))$, 
    and hence $P(\ell)=\exp\!\big(\exp(O(\ell^2\log\ell))\big)$.

    It remains to show that these choices give the full-rank conclusion. Let $p>P(\ell)$ and $q\neq p$ be primes, and let $\lambda\in\overline{\mathbb{F}}_q^\times$ have multiplicative order $p$. If $\ell=1$, the matrix $A_{1, M}(\lambda)$ has the single entry $\lambda^2-1$, which is nonzero because $p>2$. Suppose henceforth that $\ell\geq2$ and, for a contradiction, that $A_{\ell,M}(\lambda)$ has rank less than $\ell$. Then all of its $\ell\times\ell$ minors vanish, so $\lambda$ is a common root of $\overline{\Delta_1},\ldots,\overline{\Delta_t}$.

    If $q\nmid B$, Lemma \ref{lem:quantitativeGoodCharacteristics} gives $p\leq d_\ell+1\leq P(\ell)$, a contradiction. If $q\mid B$, Lemma \ref{lem:quantitativeExceptionalCharacteristics} gives $p\leq P(\ell)-1$, again a contradiction. Therefore $A_{\ell,M}(\lambda)$ has rank $\ell$, proving the proposition.
\end{proof}

\section{Proof of Theorem \ref{main thm}}\label{sec:proofMain}


Let $\mu $ be a distribution on \( \mathbb{F}_p^{k } \) whose support is the set of $k$-tuples $(x,x+y,x+y^2, \ldots, x+y^{k-1})$ with $x \in \mathbb{F}_p$ and $y \in \{0,2,3, \ldots, M_k\}$, with $M_k$ and the prime $p \geq P_k$ as in \eqref{e;MkPk}. The specific weights are fixed in Lemma \ref{l;large exp}. A point of $(\F_p^n)^k$ is written $(x_0,\ldots,x_{k-1})$ with $x_j\in\F_p^n$, and $x_{j,i}\in\F_p$ denotes the $i$-th coordinate of $x_j$; subscripts are used for coordinates because superscripts denote powers. To estimate the number of restricted polynomial progressions in a set $A \subseteq \F^n_p$, we consider the expectation
\begin{equation}
    \Expectation_{(x_0, \ldots, x_{k-1}) \sim \mu^{\otimes n}} 
    \prod_{j=0}^{k-1} 1_A(x_j),
\end{equation}
where $1_A\colon \F^n_p \rightarrow \{0,1\}$ denotes the indicator function of a set $A \subseteq \F^n_p$. 

We call $(x_0,x_1, \ldots, x_{k-1})\in(\F^n_p)^k$ a \textit{valid $k$-tuple} if the $k$ vectors $x_0, \ldots, x_{k-1}$ are not all equal and, for every coordinate $i \in [n]$, the $k$-tuple $(x_{0,i},x_{1,i}, \ldots, x_{k-1,i})$ is in the support of $\mu$. Equivalently, it is a nontrivial restricted $k$-term polynomial progression, with base point $x = x_0$ and step $y \in Y^n\setminus\{\vec 0\}$. Let $\ap = \frac{|A|}{p^n}$ and $f_A (x) = 1_A(x) - \ap$. 

We use the following two lemmas to convert stability bounds for the balanced function $f_A$ into density increment bounds for the associated set $A \subseteq \F^n_p$. 
\begin{lemma}
\label{l;stability}  Let \( \Sigma \) be a finite alphabet with measure \( \pi \). 
For every \( 0< \delta < \frac{1}{2} \), \( n \in \N \), and \(  f \colon \Sigma^n\to  \{0,1\}\),  if \( \operatorname{Stab}_{1- \delta}(f - \mathbb{E}f) > \delta \), then 
\( W_{\leq d} (f- \E f) \geq \frac{\de}{2} \) for every integer \( d \geq \frac{1}{\de} \log \frac{2}{\de} \).
\end{lemma}

\begin{proof}
    To prove the contrapositive, assume 
    \begin{align*}
         W_{\leq d} (f - \E f)  & = \sum_{\substack{\# \xi \leq d \\ \xi \neq 0}} |\wh{f}(\xi)|^2 < \frac{\de}{2}. 
     \end{align*}
     To show that the balanced function is sufficiently stable, consider
     \begin{align*}
         \Stab_{1- \de} (f - \E f)  & = \sum_{\xi \in \{0,\ldots,|\Sigma|-1\}^n} (1- \de)^{\# \xi} \big|\wh{(f- \E f)} (\xi)\big|^2 \\
         &= \sum_{\xi \neq 0 } (1- \de)^{\# \xi} |\wh{f} (\xi)|^2 \\
         &\leq \sum_{\substack{\xi \neq 0; \\ \# \xi \leq d }} (1- \de)^{\# \xi} |\wh{f} (\xi)|^2  + (1- \de)^{d+1 }\sum_{\substack{\xi \neq 0; \\ \# \xi > d }}  |\wh{f} (\xi)|^2. 
     \end{align*}
     By assumption, the function $f- \E f$ does not carry much of its norm on low degree terms. Hence, the right-hand side of the inequality above is 
     \begin{align*}
         &\leq \frac{\de}{2} + (1- \de)^{d+1 } \sum_{\substack{\xi \neq 0; \\ n \geq  \# \xi > d }}  |\wh{f} (\xi)|^2 .
     \end{align*}
    By Plancherel, $\sum_{\xi\neq 0}|\wh{f}(\xi)|^2 = \norm{f-\E f}{2}^2 = \operatorname{Var}(f) \leq \E f^2 \leq 1$, because $f$ takes values in $\{0,1\}$. Together with $1 - \de \leq e^{-\de}$, this shows that the right-hand side of the inequality above is 
    \begin{align*}
         & \leq \frac{\de}{2} + e^{-\de d} \norm{f-\E f}{2}^2 \leq \frac{\de}{2} + e^{-\de d}. 
     \end{align*}
     Since $d \geq \de^{-1} \log (2/\de)$, we have $e^{-\de d} \leq \frac{\de}{2}$, and we conclude that 
     \begin{align*}
         \Stab_{1- \de} (f - \E f) & \leq \frac{\de}{2} + \frac{\de}{2}  \leq \de. 
     \end{align*}
\end{proof}

Given a function $f: \F^{n}_p \rightarrow \C$, a \textit{restriction} of it amounts to choosing an index set $I \subseteq [n]$ and setting the associated coordinates all equal to a point $u \in \F^{I}_p$ while keeping all other coordinates unchanged. We will denote such a restriction of $f$ as follows 
\begin{align*}
    f_{I \rightarrow u}(y)= f(x_I = u, x_{\Ib} = y). 
\end{align*}
In Lemmas \ref{p;restriction}, \ref{l;fiber} and \ref{l;increment}, and in the proof of Theorem \ref{main thm}, the set $S\subseteq[n]$ (or $S_j$) denotes the \emph{free} coordinates and $\Sb=[n]\setminus S$ the complementary set of \emph{frozen} coordinates; thus, for a frozen point $u\in\Sigma^{\Sb}$, the restriction $f_{\Sb\rightarrow u}$ is a function on $\Sigma^{S}$.

\begin{lemma}[cf. {\cite[Lemma 2.4]{BKM-EffectiveRestricted}}] \label{p;restriction}
    Let $\Sigma$ be a finite set with a probability measure $\pi$. Let $f: \Sigma^n \rightarrow \{0, 1\}$, let $d \geq 2$ be an integer, and let $\beta > 0$ be such that $W_{\leq d} (f- \E f) \geq \beta$ and
    \begin{align*}
        \beta \geq 16 \exp \lp - \frac{n}{2d^2} \rp.
    \end{align*}
    Then there is $S \subseteq [n]$ of size $|S| \geq \frac{n}{2d}$ and $u \in \Sigma^{\Sb}$ such that
    \begin{align*}
        \Expectation_{\pi^{\otimes S}} [f_{\Sb \rightarrow u}] \geq \Expectation [f] + \frac{\beta}{16}.
    \end{align*}
\end{lemma}

\begin{proof}
    Let $\ap = \E [f]$. We choose a random restriction $(S, u)$ as follows: the set $S \subseteq [n]$ contains each $i \in [n]$ independently with probability $\frac{1}{d}$, and, given $S$, the point $u \sim \pi^{\otimes \Sb}$.

    \emph{The size of $S$.} The random variable $|S|$ is a sum of $n$ independent Bernoulli random variables with parameter $\frac{1}{d}$, so $\E|S| = \frac{n}{d}$. By Hoeffding's inequality,
    \begin{align}\label{e;hoeffding}
        \Prob \lp |S| < \frac{n}{2d} \rp \leq \Prob \lp |S| - \E|S| \leq -\frac{n}{2d} \rp \leq \exp \lp - \frac{2 (\frac{n}{2d})^2}{n} \rp = \exp \lp - \frac{n}{2d^2} \rp.
    \end{align}

    \emph{The second moment.} Let $Z_{S, u} \coloneqq \E_{\pi^{\otimes S }} f_{\Sb \rightarrow u} = \E_{x_S \sim \pi^{\otimes S}} f (x_S, u)$. For fixed $S$, the function $u \mapsto Z_{S,u}$ is the conditional expectation of $f$ given the coordinates in $\Sb$, and its Efron--Stein decomposition is $\sum_{S' \subseteq \Sb} f^{=S'}$. By orthogonality,
    \begin{align*}
        \E_{u \sim \pi^{\otimes \Sb}} Z_{S,u}^2 = \sum_{S' \subseteq \Sb} \norm{f^{=S'}}{2}^2 .
    \end{align*}
    Averaging over $S$, and using $f^{=\emptyset} = \ap$, we obtain
    \begin{align*}
        \E Z_{S,u}^2 = \sum_{S' \subseteq [n]} \Prob(S \cap S' = \emptyset) \norm{f^{=S'}}{2}^2
        \geq \ap^2 + \sum_{0 < |S'| \leq d} \lp 1 - \frac{1}{d} \rp^{|S'|} \norm{f^{=S'}}{2}^2 .
    \end{align*}
    Since $(1 - \frac{1}{d})^{|S'|} \geq (1-\frac{1}{d})^{d} \geq \frac14$ for $|S'| \leq d$ and $d \geq 2$, and since $\sum_{0<|S'| \leq d} \norm{f^{=S'}}{2}^2 = W_{\leq d}(f - \E f) \geq \beta$, we conclude that
    \begin{align*}
        \E Z_{S,u}^2 \geq \ap^2 + \frac{\beta}{4}.
    \end{align*}
    
    Since $Z_{S,u}^2$ is bounded below (it is non-negative) and bounded above by $1$, to obtain this lower bound $\ap^2 + \frac{\beta}{4}$, a significant enough amount of probability is allocated to $\ap^2 + \frac{\beta}{ 8}$. That is, 
    \begin{align*}
        \Prob \lp u, S; \hspace{.2in} Z^2_{S,u} \geq \ap^2 + \frac{\beta}{8} \rp \geq \frac{\beta}{8}. 
    \end{align*}
    Indeed, consider the event $\mathcal{G} := \{S, u ; \ Z_{S,u}^2 \geq \ap^2 + \frac{\beta}{8}\}$. We know that $0 \leq Z^2_{S,u} \leq 1$. Hence, 
    \begin{align*}
        \E[Z^2_{S,u}] &= \E[Z^2_{S,u} 1_{\mathcal{G}}] + \E[Z^2_{S,u} 1_{\mathcal{G}^c}] \\
        & \leq 1 \cdot \Prob(\mathcal{G}) + \lp \ap^2 + \frac{\beta}{8} \rp \Prob(\mathcal{G}^c) .
    \end{align*}
    Using the lower bound computed earlier, we get
    \begin{align*}
        \ap^2 + \frac{\beta}{4} & \leq  \Prob(\mathcal{G}) + \lp \ap^2 + \frac{\beta}{8} \rp (1- \Prob(\mathcal{G})) \\
        \implies \frac{\beta}{8} & \leq \lp 1- \lp  \ap^2 + \frac{\beta}{8} \rp\rp \Prob(\mathcal{G})
    \end{align*}
    where $\ap^2 + \frac{\beta}{8} \leq 1$. We thus obtain 
    \begin{align*}
        \frac{\beta}{8} & \leq \Prob(\mathcal{G}) 
    \end{align*}
    as claimed. 
    By \eqref{e;hoeffding}, $|S| \geq \frac{n}{2d}$ with probability at least $1- \exp\lp - \frac{n}{2d^2} \rp$. Since $\Prob(E \cap F) \geq \Prob(E) + \Prob(F) - 1$ for any two events, the hypothesis $\exp\lp - \frac{n}{2d^2} \rp \leq \frac{\beta}{16}$ gives
    \begin{align*}
        \Prob \lp \mathcal{G} \cap \left\{ |S| \geq \frac{n}{2d} \right\} \rp \geq \frac{\beta}{8} - \exp\lp - \frac{n}{2d^2} \rp \geq \frac{\beta}{16} > 0 .
    \end{align*}
    We may therefore fix $(S,u)$ for which both events hold. Since $Z_{S,u} \geq 0$, on the event $\mathcal{G}$ we have $Z_{S,u} \geq \sqrt{\ap^2 + \frac{\beta}{8}}$. Note that $\ap^2 + \frac{\beta}{8} \leq \E[Z^2_{S,u}] \leq 1$. For $0 \leq \ap \leq 1$ and $t \geq 0$ with $\ap^2 + t \leq 1$,
    \begin{align*}
        \sqrt{\ap^2 + t} - \ap = \frac{t}{\sqrt{\ap^2 + t} + \ap} \geq \frac{t}{2},
    \end{align*}
    because both terms in the denominator are at most $1$. Taking $t = \frac{\beta}{8}$ gives
    \begin{align*}
        \Expectation_{\pi^{\otimes S}} [f_{\Sb \rightarrow u}] = Z_{S,u} \geq \ap + \frac{\beta}{16},
    \end{align*}
    with $|S| \geq \frac{n}{2d}$, which concludes the proof.
\end{proof}

\begin{lemma}
\label{l;large exp}
    Let $\mu$ be the distribution over $\F_p^k $ with mass $ \frac{1 }{2p} $ on $(a,\ldots, a)$ for every $a \in \F_p$, and $\frac{ 1}{2p (M_k -1)} $ on each element $\vec{a}$, where $\vec{a} \in \supp \mu \setminus \{(a,\ldots ,a) : a \in \F_p\}$. That is, $\mu = \frac{1}{2} \mu_{\diag } + \frac{1}{2} \mu_{\off}$, where $\mu_{\diag }$ and $\mu_{\off }$ are uniformly distributed conditional on being on the diagonal terms and on the off diagonal terms respectively. If $A$ has no valid $k$-tuple, then there are an index $0 \leq j \leq k-1$ and $1$-bounded functions $f_i: \F_p^n \to \C$, $i \neq j$, such that
\begin{align} 
    \left|\Expectation_{(x_0, \ldots, x_{k-1}) \sim \mu^{\otimes n}}\Big[f_A(x_j)\prod_{\substack{i\neq j \\ 0 \leq i \leq k-1}} f_i(x_i)\Big]\right|  \geq \frac{\ap^k}{2^k-1}  - \frac{\ap}{2^n (2^k -1)}. 
\end{align}
In addition, for each $0\leq j\leq k-1$ the marginal $\mu_j$ of $\mu$ on the $j$-th coordinate is the uniform measure on $\F_p$. Consequently the marginal of $\mu^{\otimes n}$ on the $j$-th block $x_j\in\F_p^n$ is the product measure $\mu_j^{\otimes n}$, which is the uniform measure on $\F_p^n$; in particular, $\E_{x \sim \mu_j^{\otimes n}}[1_A(x)] = \ap$.
\end{lemma}

\begin{proof}
    Suppose $A$ has no valid $k$-tuple, so that the only tuples contributing to the expectation below are those with $x_0 = \cdots = x_{k-1}$. A point of $\supp\mu$ has all of its entries equal only when $y=0$, so such a tuple occurs exactly when every one of the $n$ coordinates is diagonal, an event of $\mu^{\otimes n}$-probability $2^{-n}$; conditionally on it, $x_0$ is uniform on $\F_p^n$. Hence 
    \begin{equation}\label{eq exp A}
        \begin{split}
            \Expectation_{x \sim \mu^{\otimes n}}  \prod_{j=0}^{k-1} 1_A(x_j ) 
            = \frac{\ap}{2^n}. 
        \end{split}
    \end{equation}
    Writing $1_A = \ap + f_A$ and distributing the resulting factors, we have
    \begin{align*}
        \Expectation_{\mu^{\otimes n}}\prod_{j=0}^{k-1} 1_A(x_j) & = \Expectation_{\mu^{\otimes n}}\prod_{j=0}^{k-1} (\ap + f_A)(x_j)  = \ap^k + \sum_{l=1}^{2^k -1} \Theta_l,
    \end{align*}
    where $\Theta_l$ is the expectation of a product of $k$ functions with the balanced function $f_A$ appearing at least once. By the triangle inequality and \eqref{eq exp A}, 
    \begin{align*}
        \ap^k &\leq \lf\Expectation_{\mu^{\otimes n}}\prod_{j=0}^{k-1} 1_A(x_j) \rf  + \sum_{l=1}^{2^k -1} |\Theta_l|  \leq \frac{\ap}{2^n}  + \sum_{l=1}^{2^k -1} |\Theta_l|.
    \end{align*}
    By the pigeonhole principle, the desired estimate follows. 
    Namely, for some \( 1 \leq l \leq 2^k -1 \), we have 
    \begin{align*}
    |\Theta_l| \geq \frac{\ap^k}{2^k -1} - \frac{\ap}{2^n (2^{k} - 1) }. 
    \end{align*}
    And \( \Theta_l  \) is an expectation of the desired form: one of the \( k \) 
    functions in the expectation equals \( f_A  \), while the remaining functions are all bounded by \( 1 \). 
    Finally, we proceed and compute the marginal distribution for the distribution 
    \begin{align*}
        \mu(x_0, \ldots, x_{k-1}) &= \frac{1}{2 p} 1_{\textrm{diagonal}}(x_0, \ldots, x_{k-1}) + \frac{1 }{2 p (M_k -1)} 1_{\textrm{off diagonal}}(x_0, \ldots, x_{k-1}). 
    \end{align*}
    The marginal on the $0$-th coordinate, evaluated at a point $a$ in $\F_p$, is given by 
    \begin{align*}
        \mu_{0} (a)&= \sum_{\substack{(x_1, \ldots, x_{k-1}) \in \F^{k-1}_p \\ (x_0, \ldots, x_{k-1}) \in \supp \mu }} \frac{1}{2 p}  1_{\textrm{diagonal}} (a,x_1, \ldots, x_{k-1}) + \frac{1 }{2 p(M_k -1)}  1_{\textrm{off diagonal}}(a,x_1, \ldots, x_{k-1}) \\
        &= \frac{1}{2p}  + \frac{ 1 }{2 p (M_k -1)} (M_k -1) = \frac{1}{p} . 
    \end{align*}
    The same computation applies to the $j$-th coordinate for $1\leq j\leq k-1$. The diagonal tuple through $a$ contributes $\frac{1}{2p}$. For each $y\in\{2,\ldots,M_k\}$ there is exactly one off-diagonal tuple with $x_j=a$, namely the one with $x=a-y^j$, and these contribute $\frac{1}{2p}$ in total. Hence $\mu_j(a)=\frac1p$ for all $j$. Since the coordinates of a sample from $\mu^{\otimes n}$ are independent, the block $x_j$ has law $\mu_j^{\otimes n}$, the uniform measure on $\F_p^n$, and so $\E_{x \sim \mu_j^{\otimes n}} 1_A(x) = \ap$. 
\end{proof}

\begin{lemma}[Restrictions preserve the hypothesis]\label{l;fiber}
    Let $A \subseteq \F_p^n$ contain no valid $k$-tuple. Then for every nonempty $S \subseteq [n]$ and every $u \in \F_p^{\Sb}$, the set
    \begin{align*}
        A_{\Sb \rightarrow u} \coloneqq \{ v \in \F_p^S : (x_S = v, \ x_{\Sb} = u) \in A \} \subseteq \F_p^S
    \end{align*}
    contains no valid $k$-tuple.
\end{lemma}

\begin{proof}
    Suppose that $(v_0, \ldots, v_{k-1})$ is a valid $k$-tuple in $A_{\Sb \rightarrow u}$, and define $x_j \in \F_p^n$ by $(x_j)_S = v_j$ and $(x_j)_{\Sb} = u$ for $0 \leq j \leq k-1$. Each $x_j$ lies in $A$. For a coordinate $i \in S$, the $i$-th coordinates of $x_0, \ldots, x_{k-1}$ form a tuple in $\supp \mu$ by assumption. For $i \in \Sb$ they form the diagonal tuple $(u_i, \ldots, u_i)$, which lies in $\supp \mu$ because $0 \in Y$ (take $x = u_i$ and $y = 0$). Finally, the $x_j$ are not all equal, since the $v_j$ are not. Hence $(x_0, \ldots, x_{k-1})$ is a valid $k$-tuple in $A$, a contradiction.
\end{proof}

\begin{remark}
    Lemma \ref{l;fiber} is the reason $0$ is included in $Y$: coordinates frozen by a restriction are filled in with the step $y_i = 0$. In terms of progressions, a nontrivial restricted progression $v, v + P_1(y'), \ldots, v+ P_{k-1}(y')$ in $A_{\Sb \rightarrow u}$, with $y' \in Y^S \setminus \{\vec{0}\}$, lifts to the nontrivial restricted progression in $A$ with base point $(v,u)$ and step $(y', \vec{0}) \in Y^n \setminus \{\vec{0}\}$.
\end{remark}

Theorem \ref{t;inverse} requires the dimension to be large: $n\geq n_{\ast}(\gamma,\ep)$. We first remove that requirement, so that the density increment below is available in every dimension. Throughout, $N_\rho$ denotes the noise operator of Section \ref{sec:notation} in whichever dimension is under discussion.

\begin{proposition} \label{p;allDimensions}
    Let $\mu$ be a distribution over $\Sigma_1 \times \cdots \times \Sigma_k$ that admits no nontrivial abelian embedding and in which every atom has mass at least $\gamma$. Let $\ep>0$, and let $\de = \de(\gamma,\ep)>0$ be as in Theorem \ref{t;inverse}. Then for \emph{every} dimension $n \geq 1$ the following holds: if $1$-bounded functions $f_i \colon \Sigma_i^n \to \C$, $1\leq i\leq k$, satisfy \eqref{B;bigger}, then $\Stab_{1-\de}(f_i) > \de$ for every $i = 1,\ldots,k$.
\end{proposition}

\begin{proof}
    This is a padding argument. Let $n_{\ast} = n_{\ast}(\gamma, \ep)$ be the dimension threshold of Theorem \ref{t;inverse}. For $n \geq n_{\ast}$ the assertion is Theorem \ref{t;inverse} itself, so we may assume $1 \leq n < n_{\ast}$. Let $f_i \colon \Sigma_i^n \to \C$, $1\leq i \leq k$, be $1$-bounded functions satisfying the hypothesis \eqref{B;bigger}. Pad them with $n_{\ast} - n$ dummy coordinates: define $F_i \colon \Sigma_i^{n_{\ast}} \to \C$ by
    \begin{align*}
        F_i(u_1,\ldots, u_{n_{\ast}}) \coloneqq f_i(u_1,\ldots, u_n), \qquad 1\leq i \leq k,
    \end{align*}
    so that $F_i$ depends only on the first $n$ coordinates, and is again $1$-bounded.

    The functions $F_i$ satisfy \eqref{B;bigger} as well. Indeed, under $\mu^{\otimes n_{\ast}}$ the first $n$ coordinates have law $\mu^{\otimes n}$, and the remaining $n_{\ast}-n$ coordinates, on which no $F_i$ depends, are simply integrated out. Hence
    \begin{align*}
        \Expectation_{(x_1,\ldots,x_k)\sim \mu^{\otimes n_{\ast}}} \prod_{i=1}^k F_i(x_i) 
        = \Expectation_{(x_1,\ldots,x_k)\sim \mu^{\otimes n}} \prod_{i=1}^k f_i(x_i),
    \end{align*}
    and the absolute value of the left-hand side is at least $\ep$. Since the dimension is now $n_{\ast}$, Theorem \ref{t;inverse} applies and gives $\Stab_{1-\de}(F_i) > \de$ for every $i$.

    Finally, stability is unchanged by the padding. Since $F_i$ does not depend on the coordinates outside $[n]$, its Efron--Stein decomposition with respect to $\mu_i^{\otimes n_{\ast}}$ satisfies
    \begin{align*}
        F_i^{=S} = f_i^{=S} \quad \text{for } S \subseteq [n], \qquad F_i^{=S} = 0 \quad \text{otherwise},
    \end{align*}
    and $\norm{F_i^{=S}}{2} = \norm{f_i^{=S}}{2}$ for $S\subseteq[n]$, because the padded coordinates are averaged against a probability measure. Indeed, for $S\not\subseteq [n]$ pick $j \in S$ with $j>n$; the projection formula for $F_i^{=S}$ contains the difference of two conditional expectations that differ only in whether the $j$-th coordinate is fixed, and these agree since $F_i$ is constant in that coordinate. By the formula $\Stab_\rho(g) = \sum_S \rho^{|S|}\norm{g^{=S}}{2}^2$,
    \begin{align*}
        \Stab_{1-\de}(F_i) = \sum_{S \subseteq [n]} (1-\de)^{|S|} \norm{f_i^{=S}}{2}^2 = \Stab_{1-\de}(f_i).
    \end{align*}
    Hence $\Stab_{1-\de}(f_i) > \de$ for every $i$, as claimed.
\end{proof}

We now fix the parameters of the density increment. Given $0 < \ap_0 \leq 1$, set
\begin{align}\label{e;eps0}
    \ep_0 \coloneqq \frac{\ap_0^k}{2(2^k-1)}.
\end{align}
The distribution $\mu$ of Lemma \ref{l;large exp} admits no nontrivial abelian embedding (Proposition \ref{prop:noAbelianEmbeddings} with $\ell = k-1$; see Remark \ref{r;conditions}), and each of its atoms has mass at least $\frac{1}{2p(M_k-1)}$. Let $\de_0 = \de(\ap_0) > 0$ be the parameter that Theorem \ref{t;inverse} provides for this $\mu$ and $\ep = \ep_0$, so that
\begin{align}\label{e;de0}
    \de_0 = \exp \lp - \exp^{(K)} \lp \lp \frac{\ap_0}{C_k }  \rp^{-b} \rp  \rp, \qquad K = k^{O(k)}.
\end{align}
Here $C_k \geq 1$ and $b = b(k,p)\geq 1$ come from Theorem \ref{t;inverse}: that theorem provides $\de = \exp(-\exp^{(K)}(\ep^{-O_\gamma(1)}))$, and $\ep_0 = \ap_0^k/(2(2^k-1))$ by \eqref{e;eps0}, so the exponent of $\ap_0$ is $O_\gamma(k)$. Both $C_k$ and $b$ depend only on $k$ and $p$, since $\gamma = \frac{1}{2p(M_k-1)}$ does, and enlarging either one only decreases $\de_0$, which is harmless. We may therefore assume $C_k \geq k+2$ and $b\geq 1$.
By the formula $N_\rho f = \sum_{S} \rho^{|S|} f^{=S}$, we have $\Stab_\rho(f) = \sum_S \rho^{|S|} \|f^{=S}\|_2^2$, which is nondecreasing in $\rho$. Hence the conclusion $\Stab_{1-\de_0}(f) > \de_0$ of Theorem \ref{t;inverse} persists if $\de_0$ is replaced by any smaller positive number, and we may assume $\de_0 < \frac{1}{2}$. Finally, set
\begin{align}\label{e;d0N}
    d \coloneqq \max \left\{ 2, \ \left\lceil \frac{1}{\de_0} \log \frac{2}{\de_0} \right\rceil \right\}, \qquad
    N(\ap_0) \coloneqq \max \left\{\ \left\lceil 2d^2 \log \frac{32}{\de_0} \right\rceil, \ \left\lceil \log_2 \frac{1}{\ap_0^{k-1}} \right\rceil + 1 \right\},
\end{align}
where $\log$ denotes the natural logarithm. All of these parameters depend only on $\ap_0$, $k$ and $p$; in particular, they do not change along the iteration below.

\begin{lemma}[Density increment]\label{l;increment}
    Let $0 < \ap_0 \leq 1$, and let $\de_0$, $d$ and $N(\ap_0)$ be as in \eqref{e;de0} and \eqref{e;d0N}. Suppose that $A \subseteq \F_p^n$ contains no valid $k$-tuple, has density $\ap = \frac{|A|}{p^n} \geq \ap_0$, and that $n \geq N(\ap_0)$. Then there exist $S \subseteq [n]$ with $|S| \geq \frac{n}{2d}$ and $u \in \F_p^{\Sb}$ such that
    \begin{align*}
        \frac{|A_{\Sb \rightarrow u}|}{p^{|S|}} \geq \ap + \frac{\de_0}{32}.
    \end{align*}
\end{lemma}

\begin{proof}
    Since $n \geq \left\lceil \log_2 \ap_0^{1-k} \right\rceil + 1$ and $\ap \geq \ap_0$, we have $2^{n-1} \geq \ap^{1-k}$, that is, $\frac{\ap}{2^n} \leq \frac{\ap^k}{2}$. By Lemma \ref{l;large exp}, there exist an index $0 \leq j \leq k-1$ and $1$-bounded functions $f_i$, $i \neq j$, such that
    \begin{align*}
        \left| \Expectation_{(x_0, \ldots, x_{k-1}) \sim \mu^{\otimes n}} \Big[ f_A(x_j) \prod_{\substack{i\neq j \\ 0 \leq i \leq k-1}} f_i(x_i) \Big] \right|
        \geq \frac{\ap^k}{2^k-1} - \frac{\ap}{2^n(2^k-1)} \geq \frac{\ap^k}{2(2^k-1)} \geq \ep_0 .
    \end{align*}
    The balanced function $f_A = 1_A - \ap$ is $1$-bounded, so Theorem \ref{t;inverse}, in the form of Proposition \ref{p;allDimensions}, gives
    \begin{align*}
        \Stab_{1-\de_0}(f_A) > \de_0 .
    \end{align*}
    By Lemma \ref{l;stability}, applied with $\de = \de_0$ and our choice of $d$, the balanced function satisfies \(  W_{\leq d}(f_A) \geq \frac{\de_0}{2} \).  
    A substantial part of its \( L^2 \) norm is carried on low degree terms. 
    
    We now apply Lemma \ref{p;restriction} on $\Sigma = \F_p$ equipped with the uniform measure, which is the common marginal of $\mu$ (Lemma \ref{l;large exp}), with $\beta = \frac{\de_0}{2}$. Its hypothesis $\beta \geq 16 \exp \lp -\frac{n}{2d^2} \rp$ is equivalent to $n \geq 2d^2 \log \frac{32}{\de_0}$, which holds since $n \geq N(\ap_0)$. The lemma yields $S \subseteq [n]$ with $|S| \geq \frac{n}{2d}$ and $u \in \F_p^{\Sb}$ such that the density of $A_{\Sb \rightarrow u}$ in $\F_p^S$ is at least $\ap + \frac{\beta}{16} = \ap + \frac{\de_0}{32}$.
\end{proof}

\begin{proof}[Proof of Theorem \ref{main thm}]
    Let $A \subseteq \F_p^n$ contain no nontrivial restricted $k$-term polynomial progression, that is, no valid $k$-tuple, and let $\ap_0 = \frac{|A|}{p^n} > 0$. Let $\de_0$, $d$ and $N(\ap_0)$ be as in \eqref{e;de0} and \eqref{e;d0N}, and let $J \coloneqq \lfloor 32/\de_0 \rfloor$.

    \emph{The iteration.} Set $A_0 = A$ and $n_0 = n$. Suppose that $A_j \subseteq \F_p^{n_j}$ has been constructed, contains no valid $k$-tuple, and has density $\ap_j \geq \ap_0 + j \frac{\de_0}{32}$. If $n_j \geq N(\ap_0)$, Lemma \ref{l;increment} (which applies since $\ap_j \geq \ap_0$) gives $S_j \subseteq [n_j]$ with $|S_j| \geq \frac{n_j}{2d}$ and a point $u_j \in \F_p^{S_j^{c}}$. We set $A_{j+1} \coloneqq (A_j)_{S_j^{c} \rightarrow u_j} \subseteq \F_p^{S_j}$ and $n_{j+1} \coloneqq |S_j|$. By Lemma \ref{l;fiber}, $A_{j+1}$ contains no valid $k$-tuple, and its density satisfies $\ap_{j+1} \geq \ap_j + \frac{\de_0}{32} \geq \ap_0 + (j+1)\frac{\de_0}{32}$.

    \emph{Termination.} Densities never exceed $1$, while $\ap_0 + (J+1)\frac{\de_0}{32} > 1$. Hence the procedure cannot complete $J+1$ steps, so there is some $j \leq J$ with $n_j < N(\ap_0)$. Since $n_{i+1} \geq \frac{n_i}{2d}$ for every $i < j$,
    \begin{align}\label{e;nbound}
        n \leq (2d)^j n_j < (2d)^{\frac{32}{\de_0}} N(\ap_0).
    \end{align}

    \emph{Quantitative bound.} Since $\de_0 < \frac12$, we have $d \leq \de_0^{-2}$. Recall from \eqref{e;de0} that $C_k \geq k+2$ and $b \geq 1$. Then \eqref{e;de0} gives $\de_0^{-1}\geq \exp\big((\ap_0/C_k)^{-1}\big) \geq \frac{C_k}{\ap_0}$, while $\left\lceil \log_2 \ap_0^{1-k} \right\rceil + 1 \leq k\log_2\frac{1}{\ap_0}+2 \leq \frac{k+2}{\ap_0}$. Hence $\left\lceil \log_2 \ap_0^{1-k} \right\rceil + 1 \leq \de_0^{-1}$. Moreover, $2d^2 \log \frac{32}{\de_0} \leq C \de_0^{-5}$. 
     Hence $N(\ap_0) \leq C \de_0^{-5}$, and \eqref{e;nbound} gives
    \begin{align*}
        n < C \de_0^{-5} \lp 2 \de_0^{-2} \rp^{\frac{32}{\de_0}} \leq \exp \lp C \de_0^{-2} \rp .
    \end{align*}
    Inserting \eqref{e;de0}, we have $\de_0^{-2} = \exp \lp 2 \exp^{(K)} \lp (\ap_0/C_k)^{-b} \rp \rp$, and therefore, after absorbing constants depending only on $k$ and $p$,
    \begin{align*}
        n < \exp^{(K+2)} \lp \lp \frac{\ap_0}{C_k} \rp^{-b} \rp .
    \end{align*}
    Taking $K+2$ logarithms, which is possible once $n$ is large enough that $\log^{(K+2)} n$ is defined (for smaller $n$ the theorem holds trivially after enlarging $C$), we obtain
    \begin{align*}
        \ap_0 \leq C_k \lp \log^{(K+2)} n \rp^{- 1/b } .
    \end{align*}
    Since $K + 2 = k^{O(k)}$, and $C_k$ and $b$ depend only on $k$ and $p$, this is the bound claimed in Theorem \ref{main thm}, with $C = C_k$ and $c = 1/b$.
\end{proof}

\printbibliography[title={References}]

\end{document}